\documentclass[aop,MSNbibl,citesort,dvips]{arximspdf}
\usepackage{graphicx}

\doi{10.1214/11-AOP731} 
\volume{41}
\issue{1}
\pubyear{2013}
\firstpage{229}
\lastpage{261}

\makeatletter

\newtheorem{theorem}{Theorem}[section]
\newtheorem{lemma}[theorem]{Lemma}
\newtheorem{corollary}[theorem]{Corollary}
\newtheorem{prop}[theorem]{Proposition}

\newtheorem{claim}[theorem]{Claim}

\newproclaim{defn}[theorem]{Definition}

\newcommand{\xrightarrow}[1]{\displaystyle\mathop{\hbox to 1cm{\rightarrowfill}}_{#1}}

\newcommand{\underset}[2]{\displaystyle\mathop{#2}_{#1}}

\newcommand{\pp}{{\mathbb P}}
\newcommand{\zz}{{\mathbb Z}}
\newcommand{\nn}{{\mathbb N}}
\newcommand{\rr}{{\mathbb R}}
\newcommand{\veps}{\varepsilon}

\newcommand{\Z}{{\mathbb Z}}
\newcommand{\N}{{\mathbb N}}
\newcommand{\E}{{\mathbb E}}
\newcommand{\R}{{\mathbb R}}
\renewcommand{\P}{{\mathbb P}}

\newcommand{\cA}{{\mathcal A}}
\newcommand{\cC}{{\mathcal C}}
\newcommand{\cF}{{\mathcal F}}
\newcommand{\cG}{{\mathcal G}}
\newcommand{\cZ}{{\mathcal Z}}
\newcommand{\cT}{{\mathcal T}}
\newcommand{\cU}{{\mathcal U}}
\newcommand{\cY}{{\mathcal Y}}
\newcommand{\cV}{{\mathcal V}}
\newcommand{\cR}{{\mathcal R}}
\newcommand{\cE}{{\mathcal E}}
\newcommand{\cP}{{\mathcal P}}
\newcommand{\cS}{{\mathcal S}}
\newcommand{\eps}{\varepsilon}

\newcommand{\bT}{{\mathbf T}}
\newcommand{\bl}{{\mathbf l}}
\newcommand{\bV}{{\mathbf V}}
\newcommand{\bx}{{\mathbf x}}
\newcommand{\by}{{\mathbf y}}
\newcommand{\bG}{{\mathbf G}}

\newcommand{\BB}{\mbox{BB}}
\newcommand{\IIC}{\mbox{IIC}}
\newcommand{\IPC}{\mbox{IPC}}

\newcommand{\Bin}{\operatorname{Bin}}

\makeatother

\begin{document}
\begin{frontmatter}

\title{Scaling limit of the invasion percolation cluster on a regular tree}
\runtitle{Scaling limit of invasion percolation}

\begin{aug}
\author[A]{\fnms{Omer} \snm{Angel}\thanksref{OAmark}\ead[label=OAemail]{angel@math.ubc.ca}},
\author[B]{\fnms{Jesse} \snm{Goodman}\corref{}\thanksref{OAmark}\ead[label=JGemail]{goodmanja@math.leidenuniv.nl}}
\and
\author[C]{\fnms{Mathieu} \snm{Merle}\thanksref{MMmark}\ead[label=MMemail]{mathieu.merle@math.univ-paris-diderot.fr}}
\runauthor{O. Angel, J. Goodman and M. Merle}
\affiliation{University of British Columbia, Univertiteit Leiden and
Universit\'e~Paris~Diderot}
\address[A]{O. Angel\\
Department of Mathematics\\
University of British Columbia\\
1984 Mathematics Road\\
Vancouver, British Columbia\\
V6T 1Z2\\
Canada\\
\printead{OAemail}}
\address[B]{J. Goodman\\
Mathematisch Instituut\\
Universiteit Leiden\\
PO Box 9512, 2300 RA Leiden\\
Netherlands\\
\printead{JGemail}}
\address[C]{M. Merle\\
Universit\'e Paris Diderot\\
175, rue du Chevaleret\\
75013 Paris\\
France\\
\printead{MMemail}} 
\end{aug}

\thankstext{OAmark}{Supported in part by NSERC.}

\thankstext{MMmark}{Supported in part by the Pacific Institute for the
Mathematical Sciences.}

\received{\smonth{10} \syear{2009}}
\revised{\smonth{10} \syear{2011}}

%
\begin{abstract}
We prove existence of the scaling limit of the invasion percolation
cluster (IPC) on a regular tree. The limit is a random real tree with a
single end. The contour and height functions of the limit are described
as certain diffusive stochastic processes.

This convergence allows us to recover and make precise certain
asymptotic results for the IPC. In particular, we relate the limit of
the rescaled level sets of the IPC to the local time of the scaled height
function.
\end{abstract}

%

\end{frontmatter}

\section{Introduction and main results}\label{intro}

Invasion percolation on an infinite connected graph is a random growth
model which is closely related to critical percolation, and is a prime
example of self-organized criticality. It was introduced in the
eighties by
Wilkinson and Willemsen~\cite{WilkWillem1983} and first studied on the
regular tree by Nickel and Wilkinson~\cite{NickWilk1983}. %
The relation between invasion
percolation and critical percolation has been studied by many authors (see,
e.g.,~\cite{CCN1985,Jarai2003}). 
More recently, Angel, Goodman, den
Hollander and Slade~\cite{AGdHS2008}
have given a structural representation
of the invasion percolation cluster on a regular tree, and used it to
compute the scaling limits of various quantities related to the IPC
such as
the distribution of the number of invaded vertices at a given level of the
tree.

Fixing a degree $\sigma\ge2$, we consider $\cT=\cT_\sigma$: the rooted
regular tree with index~$\sigma$, that is, the rooted tree where every
vertex has $\sigma$ children. Invasion percolation on $\cT$ is defined
as follows: edges of $\cT$ are assigned weights which are i.i.d. and
uniform on $[0,1]$. The invasion percolation cluster on $\cT$, denoted
IPC, is grown inductively starting from a subgraph $I_0$ consisting of
the root $\varnothing$ of $\cT$. At each step $I_{n+1}$ consists of
$I_n$ together with the edge of minimal weight in the boundary of
$I_n$. The invasion percolation cluster IPC is the limit $\bigcup I_n$.

\subsection{Convergence of trees}

We consider the IPC as a metric space with respect to
graph distance $d_{\mathrm{gr}}$. Since IPC is already infinite, taking
its scaling limit amounts to replacing $d_{\mathrm{gr}}$ by $\frac
{1}{k}d_{\mathrm{gr}}$.
%
\begin{theorem}\label{Tlimexist}
The rescaled rooted invasion percolation cluster $(\mbox{IPC},\break\frac{1}{k}d_{\mathrm{gr}}$,
$\varnothing)$ has a scaling limit w.r.t. the pointed
Gromov--Hausdorff topology,
which is a random $\R$-tree.
\end{theorem}

Here, an $\R$-tree means a topological space with a unique rectifiable
simple path between any two points. Note that, because the IPC is
infinite, we must work with the pointed Gromov--Hausdorff topology (see,
e.g.,~\cite{Munn2010}, Section~2). For present purposes this
means we must show that, for each $R>0$, the ball $\{v\in\mathrm
{IPC}\dvtx \frac{1}{k}d_{\mathrm{gr}}(\varnothing, x)\leq R\}$ about
the root
converges in the Gromov--Hausdorff sense.

A key point in our study is that the contour function (as well as height
function and Lukaciewicz path, see Section~\ref{subtrees} below) of
an infinite tree does not generally encode
the entire tree. If the various encodings of trees are applied to infinite
trees, they describe only the part of the tree to the left of the leftmost
infinite branch. We present two ways to overcome this difficulty. Both are
based on the fact (see~\cite{AGdHS2008})
that the IPC has a.s. a unique
infinite branch. Following Aldous~\cite{Aldous1991}, we define a
\textit{sin-tree} to be an
infinite one-ended tree (i.e., with a single infinite branch).

The first approach is to use the symmetry of the underlying graph $\cT
$ and
observe that the infinite branch of the IPC (called the \textit{backbone}) is
independent of the metric structure of the IPC. Thus, for all purposes
involving only the metric structure of the IPC, we may as well assume (or
condition) that the backbone is the rightmost branch of $\cT$. We
denote by
$\cR$ the IPC under this condition. The various encodings of $\cR$ encode
the entire tree.

The second approach is to consider a pair of encodings, one for the
part of
the tree to the left of the backbone, and a second encoding the part to the
right of the backbone. This is done by considering also the encoding of the
reflected tree $\overline{\IPC}$. The reflection of a plane tree is
defined to
be the same tree with the reversed order for the children of each vertex.
The uniqueness of the backbone implies that together the two encodings
determine the entire IPC.

In order to describe the limits, we first define the process $L(t)$
which is
the lower envelope of a Poisson process on $(\R^+)^2$. Given a Poisson
process $\cP$ of intensity 1 in the quarter plane, $L(t)$ is defined by
\[
L(t) = \inf\{ y\dvtx (x,y)\in\cP\mbox{ and } x\le t\}.
\]

Our other results describe the scaling limits of the various encodings
of the
trees in terms of solutions of
{\renewcommand{\theequation}{$\cE(L)$}
\begin{equation} \label{eqSDE}
Y_t = B_t - \int_0^t L( -\underline{Y}_{ s} ) \,ds,
\end{equation}}

\noindent where $\underline{Y}_{  s}=\inf_{0\leq u\leq s} Y_u$ is the infimum
process of $Y$ and $B_t$ is a standard Brownian motion. The reason for the
notation is that we also consider solutions of equations $\cE(L/2)$ where,
in the above, $L$ is replaced by $L/2$. Note that by the scale
invariance of
the Poisson process, $k L(kt)$ has the same law as
$L(t)$. Hence, the scaling of Brownian motion implies that the solution
$Y$ has Brownian scaling as well.

We work primarily in the space $\cC(\R^+,\R^+)$ of continuous functions
from $\R^+$ to itself with the topology of locally uniform
convergence. We
consider three well known and closely related encodings of plane trees,
namely, the Lukaciewicz path, and the contour and height functions (all are
defined in Section~\ref{subenc1} below). The three are closely
related and, indeed, their scaling
limits are almost the same. The reason for the triplication is that the
contour function is the simplest and most direct encoding of a plane tree,
whereas the Lukaciewicz path turns out to be easier to deal with in
practice. The height function is a middle ground.

%
\begin{theorem}\label{Trightscale}
For the IPC conditioned on the backbone being on the right, let $V_\cR
$, $H_\cR$ and $C_\cR$ denote its Lukaciewicz path, height function and
contour function, respectively.
Then we have the following weak limits in $\cC(\R^+,\R)$:
%
\setcounter{equation}{0}
\begin{eqnarray}
\label{equ1}
(k^{-1} V_\cR(k^2 t))_{t\ge0} &\to&
\bigl(\gamma^{1/2} (Y_{t}-\underline{Y}_{  t})\bigr)_{t\ge0}, \\
(k^{-1} H_\cR(k^2 t))_{t\ge0} &\to&
\bigl(\gamma^{-1/2} (2Y_{t}-3 \underline{Y}_{  t})\bigr)_{t\ge0}, \\
(k^{-1} C_\cR(2k^2 t))_{t\ge0} &\to&
\bigl(\gamma^{-1/2} (2Y_{t}-3 \underline{Y}_{  t})\bigr)_{t\ge0}
\end{eqnarray}
as $k\to\infty$, where
\[
\gamma= \frac{\sigma-1}{\sigma}
\]
and $(Y_t)_{t\ge0}$ is the solution of (\ref{eqSDE}) (and is the
same solution in all three limits).
\end{theorem}

To put this theorem into context, recall that the Lukaciewicz path of a
critical Galton--Watson tree is an excursion of random walk with
i.i.d. steps. From this it follows that the path of an infinite sequence
of critical trees scales to Brownian motion. The height and contour
functions of the sequence are easily expressed in terms of the Lukaciewicz
path and, assuming the branching law has\vadjust{\goodbreak} second moments, are seen to
scale to reflected Brownian motion (cf. Le Gall
\cite{LeGall2005}).
Duquesne and Le Gall generalized this approach in~\cite{DuqLeG2002},
and showed that the genealogical structure of a continuous-state branching
process is similarly coded by a height process which can be expressed in
terms of a L\'evy process, and that this is also the limit of various
Galton--Watson trees with heavy tails.

The case of sin-trees is considered by Duquesne~\cite{Duquesne2005}
to study
the scaling limit of the range of a random walk on a regular tree. His
techniques suffice for analysis of the IIC, but the IPC requires additional
ideas, the key difficulty being that the Lukaciewicz path is no longer a
Markov process. The scaling limit of the IIC turns out to be an
illustrative special case of our results, and we will describe its scaling
limit as well (in Section~\ref{secIIC}).

For the unconditioned IPC we define its left part $\IPC_{G}$ to be the subtree
consisting of the backbone and all vertices to its left. The right part
$\IPC_{D}$ is
defined as the left part of the reflected IPC. We can now define
$V_G$ and $V_D$ to be, respectively, the Lukaciewicz paths for the left and
right parts of the IPC, and similarly define $H_G,H_D,C_G,C_D$ (see
also Section~\ref{subenc2} below).
%
\begin{theorem}\label{Ttwosidedscale}
We have the following weak limits in $\cC(\R^+,\R)$:
%
\begin{eqnarray}
k^{-1}(V_G(k^2 t), V_D(k^2 t))_{t\ge0} &\to&
\gamma^{1/2}(Y_{t}-\underline{Y}_{ t}, \widetilde Y_{t}-\underline
{\widetilde Y}_{ t})_{t\ge0},\\
k^{-1}(H_G(k^2 t), H_D(k^2 t))_{t\ge0} &\to&
\gamma^{-1/2}(Y_{t}-2 \underline{Y}_{ t}, \widetilde Y_{t}-2
\underline{\widetilde Y}_{ t})_{t\ge0},\\
\label{eqtwosidedscalecontour}
k^{-1}(C_G(2k^2 t), C_D(2k^2 t))_{t\ge0} &\to&
\gamma^{-1/2}(Y_{t}-2 \underline{Y}_{ t}, \widetilde Y_{t}-2
\underline{\widetilde Y}_{ t})_{t\ge0}
\end{eqnarray}
as $k\to\infty$, where $(Y_t)_{t\ge0}$ and $(\widetilde
{Y}_t)_{t\ge0}$ are
\textit{independent} solutions of $\cE(L/2)$.
\end{theorem}

\subsection{Level sizes and volumes}

From the convergence results above we can establish
asymptotics for level sizes and volumes in the invasion percolation
cluster. In~\cite{AGdHS2008},
it was proved that the size of the $n$th level
of the IPC, rescaled by a factor $n$, converges to a nondegenerate
limit. Similarly, the volume up to level $n$, rescaled by a factor $n^2$,
converges to a nondegenerate limit. The Laplace transforms of these limits
were expressed as functions of the $L$-process. However, formulas
(1.20)--(1.23) of~\cite{AGdHS2008}
do not provide insight into the limiting variables.
With our convergence theorem for height functions of $\cR$, we can
express the limit in terms of the continuous
limiting height function.

For $x \in\rr_+$ we denote by $C[x]$ the number of vertices of the
IPC at height~$[x]$. We let $C[0,x]=\sum_{i=0}^{[x]} C[i]$ denote
the number of vertices of the IPC up to height~$[x]$. Write $H_t =
\gamma^{-1/2}(2Y_t - 3\underline{Y}_t)$ for the limit of $H_\cR$ in
Theorem~\ref{Trightscale}, and $l_\infty^a(H)$ for the standard
local time
at level $a$ of $H$.
%
\begin{theorem}\label{Tlevels-volume}
For every $a>0$ we have the distributional limits
%
\begin{equation}\label{eqvolume}
\frac{1}{n^2} C[0,an] \xrightarrow{n \to\infty}
\int_{0}^{\infty} \mathbf{1}_{[0,a]}(H_s) \,ds\vadjust{\goodbreak}
\end{equation}
and
%
\begin{equation}\label{eqlevels}
\frac{1}{n} C[an] \xrightarrow{n \to\infty}
\frac{\gamma}{4} l_{\infty}^a(H).
\end{equation}
\end{theorem}

In the case of the asymptotics of the levels, we also provide an
alternative way of expressing the limit directly as a sum of
independent variables. Write $\mathbf{e}\{c\}$ for an exponential
variable of rate $c$.
%
\begin{theorem}\label{Tlevels}
Let $S$ be a point process such that, conditioned on the $L$-process, $S$
is an inhomogeneous Poisson point process on $[0, a \sqrt{\gamma}]$, with
intensity
\[
\frac{2 L(s)\,ds} {\exp( (a \sqrt{\gamma} - s)
L(s) ) - 1}.
\]
Then, conditionally on $L$, and in distribution,
%
\begin{equation}\label{eqsumofexp}
\frac{1}{n} C[an] \xrightarrow{n\to\infty}
\frac{\sqrt{\gamma}}{2} \sum_{s\in S} \mathbf{e}\biggl\{ \frac
{L(s)} {1-
\exp(-(a \sqrt{\gamma} -s)L(s))} \biggr\},
\end{equation}
where the terms in the sum are independent.
\end{theorem}

From this representation and properties of the $L$-process, it is
straightforward to recover the representation of the asymptotic Laplace
transform of level sizes, (1.21) of~\cite{AGdHS2008}.
Also, as the proof of the theorem will show, a.s. only a finite number
of distinct values of $L$ contribute to the sum in
(\ref{eqsumofexp}).

\subsection{Application to the incipient infinite cluster}

The proofs of Theorems~\ref{Tlimexist}--\ref{Tlevels} also apply to
the \textit{incipient infinite cluster} (IIC), whose structure and
similarity to the IPC we outline in Section~\ref{subIICrecall}.
Stated briefly, the IIC corresponds to the IPC in the simpler case
where the process $L(t)$ is replaced by $0$. As a consequence, some
elements of the proofs (such as the right-grafting constructions in
Section~\ref{secproofmain}) are not needed to handle the IIC. For
comparison, we summarize the results for the IIC in the following
theorems.\looseness=-1
%
\begin{theorem}\label{TIICresults1}
The rescaled rooted incipient infinite cluster $(\mbox{IIC},\frac
{1}{k}d_{\mathrm{gr}},\varnothing)$ has a scaling limit w.r.t. the pointed
Gromov--Hausdorff topology,
which is a random $\R$-tree.

For the IIC conditioned on the backbone being on the right, let $V_\cR
^{\mathrm{IIC}}$, $H_\cR^{\mathrm{IIC}}$ and $C_\cR^{\mathrm{IIC}}$
denote its Lukaciewicz path, height function and contour function, respectively.
Then we have the following weak limits in $\cC(\R^+,\R)$:
%
\begin{eqnarray}
\label{eLukaIICRlimit}
(k^{-1} V_\cR^{\mathrm{IIC}}(k^2 t))_{t\ge0} &\to&
\bigl(\gamma^{1/2} (B_{t}-\underline{B}_{  t})\bigr)_{t\ge0}, \\
\label{eHeightIICRlimit}
(k^{-1} H_\cR^{\mathrm{IIC}}(k^2 t))_{t\ge0} &\to&
\bigl(\gamma^{-1/2} (2B_{t}-3 \underline{B}_{  t})\bigr)_{t\ge0},
\\
\label{eContourIICRlimit}
(k^{-1} C_\cR^{\mathrm{IIC}}(2k^2 t))_{t\ge0} &\to&
\bigl(\gamma^{-1/2} (2B_{t}-3 \underline{B}_{  t})\bigr)_{t\ge0}
\end{eqnarray}
as $k\to\infty$, where $B_t$ is a standard Brownian motion.\vadjust{\goodbreak}

For the IIC with unconditioned backbone, the Lukaciewicz paths, height
functions and contour functions of its left and right parts have the
following weak limits in $\cC(\R^+,\R)$:
%
\begin{eqnarray}
\label{eLukaIIClimit}
k^{-1}(V_G^{\mathrm{IIC}}(k^2 t), V_D^{\mathrm{IIC}}(k^2 t))_{t\ge
0} &\to&
\gamma^{1/2}(B_{t}-\underline{B}_{ t}, \widetilde B_{t}-\underline
{\widetilde B}_{ t})_{t\ge0}, \\
\label{eHeightIIClimit}
k^{-1}(H_G^{\mathrm{IIC}}(k^2 t), H_D^{\mathrm{IIC}}(k^2 t))_{t\ge
0} &\to&
\gamma^{-1/2}(B_{t}-2 \underline{B}_{ t}, \widetilde B_{t}-2
\underline{\widetilde B}_{ t})_{t\ge0}, \\
\label{eContourIIClimit}
k^{-1}(C_G^{\mathrm{IIC}}(2k^2 t), C_D^{\mathrm{IIC}}(2k^2 t))_{t\ge
0} &\to&
\gamma^{-1/2}(B_{t}-2 \underline{B}_{ t}, \widetilde B_{t}-2
\underline{\widetilde B}_{ t})_{t\ge0}
\end{eqnarray}
as $k\to\infty$, where $B_t$ and $\widetilde B_t$ are
\textit{independent} Brownian motions.
\end{theorem}

Note that up to constant factors, the scaling limits in (\ref
{eLukaIICRlimit}) and (\ref{eLukaIIClimit}) are reflected Brownian
motions, while the scaling limits in (\ref{eHeightIIClimit}) and (\ref
{eContourIIClimit}) are three-dimensional Bessel processes. The scaling
limit in (\ref{eHeightIICRlimit}) and (\ref{eContourIICRlimit}),
however, is not a standard process.
%
\begin{theorem}\label{TIICresults2}
Write $H_t^{\mathrm{IIC}} = \gamma^{-1/2}(2Y_t - 3\underline{Y}_t)$ for
the limit of $H_\cR^{\mathrm{IIC}}$ in (\ref{eHeightIICRlimit}), and
$l_\infty^a(H)$ for the standard local time at level $a$ of $H$. Then
for every $a>0$ we have the distributional limits
%
\begin{equation}\label{eqIICvolume}
\frac{1}{n^2} C[0,an] \xrightarrow{n \to\infty}
\int_{0}^{\infty} \mathbf{1}_{[0,a]}(H_s) \,ds
\end{equation}
and
%
\begin{equation}\label{eqIIClevels}
\frac{1}{n} C[an] \xrightarrow{n \to\infty}
\frac{\gamma}{4} l_{\infty}^a(H).
\end{equation}
Moreover, if $S^{\mathrm{IIC}}$ is an inhomogeneous Poisson point
process on $[0,a\sqrt{\gamma}]$ with intensity $2(a\sqrt{\gamma
}-s)^{-1}\,ds$, then
%
\begin{equation}
\frac{1}{n}C[an]\xrightarrow{n\to\infty}
\frac{\sqrt{\gamma}}{2} \sum_{s\in S} \mathbf{e}\bigl\{ \bigl(a\sqrt
{\gamma
}-s\bigr)^{-1} \bigr\}
\end{equation}
in distribution, where the terms in the sum are independent.
\end{theorem}

\section{Background and overview}

\subsection{Structure of the IPC}
\label{subrecall}

We now give a brief overview of the IPC structure theorem from \cite
{AGdHS2008},
which is the basis for the present work. First of all, the IPC
contains a single infinite branch, called the backbone and denoted
$\BB$. The backbone is a uniformly random branch in the tree (in the
natural sense). From the backbone emerge, at every height $n$ and on
every edge
away from the backbone, subcritical percolation clusters
with parameter $\widehat W_n< p_c=\sigma^{-1}$.

The parameters $\widehat W_n$ are nondecreasing and satisfy $\widehat
W_n\xrightarrow{n\to\infty} p_c$. Moreover, $(\widehat
W_n)_{n=0}^\infty
$ forms a Markov chain with dynamics\vadjust{\goodbreak} of the following kind. The initial
value $\widehat W_0$ is distributed\vspace*{1pt} on $[0,p_c]$ according to a certain
density function $f$. Given $\widehat W_n=\widehat w$, the next value
$\widehat W_{n+1}$ is, with probability $g(\widehat w)$, a new value
chosen according to the density $f$ conditioned to be larger than
$\widehat w$; or else, with probability $g(\widehat W_n)$, the value
$\widehat w$. For our purposes, it will suffice to know that the
functions $f$ and $g$ satisfy
%
\begin{equation}\label{eRateAsymps}
\lim_{\widehat w\nearrow p_c}f(\widehat w)>0,\qquad
g(\widehat w)\sim\sigma(p_c-\widehat w)=1-\sigma\widehat w
\end{equation}
as $\widehat w\nearrow p_c$. (These asymptotics follow from
\cite{AGdHS2008}, Sections 2.1.2 and 3.1, since $(\widehat
W_n)_{n=0}^\infty$ is the image of the Markov chain
$(W_n)_{n=0}^\infty$ under $w\mapsto \widehat w$.)

%
%
We will primarily be concerned with the scaling limit of $\widehat
W_n$, which
is given by the lower envelope process $L(t)$ defined above. Writing
$[x]$ for the integer part of~$x$, we have, for any
$\eps>0$,
%
\begin{equation}\label{eWtoL}
%
\bigl(k\bigl(1 - \sigma\widehat W_{[kt]}\bigr)\bigr)_{t \ge\eps}
\xrightarrow{k\to\infty}
(L(t))_{t \ge\eps}
\end{equation}
with respect to the Skorohod topology (see~\cite{AGdHS2008},
Proposition 3.3 and Corollary~3.4). Indeed, $L(t)$ is the continuous-time
process that jumps, at rate $L(t)$, to a value uniformly chosen between
$0$ and $L(t)$; this reflects the asymptotics given in (\ref{eRateAsymps}).

The process $L_t$ diverges as $t\to0$, which somewhat complicates the study
of the IPC close to the root.

\subsection{Structure of the IIC}
\label{subIICrecall}

The \textit{incipient infinite cluster} (IIC) embodies the notion of a
percolation cluster that is both critical and infinite. It was
originally defined and discussed by Kesten~\cite{Kesten1986} (see also
\cite{BarKum2006}). The IIC can be obtained through a variety of
limiting constructions---for instance, by conditioning a critical
percolation cluster to extend at least distance $R$ and sending $R\to
\infty$, or by examining the neighborhood of a faraway point in the IPC
(see~\cite{Jarai2003} and~\cite{AGdHS2008}, Theorem 1.2). In the
present context, we note that the IIC on a regular tree has a structure
similar to the IPC; see~\cite{AGdHS2008}, Section~2.1.

Specifically, the IIC
contains a single infinite branch, the backbone, which is a uniformly
random branch in the tree. From the backbone emerge, at every height
and on every edge
away from the backbone, \textit{critical} percolation clusters.

Note that setting $\widehat W_n\equiv p_c$ in the above description
gives rise to
the IIC, on the one hand, while in the scaling limit $L$ is replaced by 0.
This enables us to use a common framework for both clusters.

The convergence $\widehat W_n\xrightarrow{n\to\infty} p_c$ explains
why the
IPC and IIC resemble each other far above the root. However, the analysis
of~\cite{AGdHS2008}
shows that the convergence of the parameter of the
attached clusters is slow enough that $r$-point functions and other
measurable quantities such as level sizes possess different scaling limits.

\subsection{Encodings of finite trees} \label{subenc1}

For completeness we include here the definition of the various tree
encodings we are concerned with. We refer to Le Gall~\cite{LeGall2005}
for further
details in the case of finite trees and to Duquesne~\cite{Duquesne2005}
in the
case of sin-trees discussed below.

A \textit{rooted plane tree} $\theta$ (also called an ordered tree) is a tree
with a description as follows. Vertices of $\theta$ belong to $\bigcup_{n
\geq0} \N^n$. By convention, $\varnothing\in\N^0$ is always a
vertex of
$\theta$ which is called the root. For a vertex $v \in\theta$, we
let $k_v
= k_v(\theta)$ be the number of children of $v$ and whenever $k_v = k >0$,
these children are denoted $v1,\ldots,vk$. In particular, the $i$th
child of the root is simply $i$, and if $vi \in\theta$, then $\forall
1\le
j<i$, $vj \in\theta$ as well. Edges of $\theta$ are the edges $(v,vi)$
whenever $vi\in\theta$. Note that the set of edges of $\theta$ are
determined by the set of vertices and vice-versa, which allows us to blur
the distinction between a tree and its set of vertices. The \textit{$k$th
generation} of a tree contains every vertex $v \in\theta\cap\N^k$, so
that the $0$th generation consists exactly of the root. Define $\#
\theta
$ to
be the total number of vertices in $\theta$.

Let $(v^i)_{0\le i < \#\theta}$ be the vertices of $\theta$ listed
in lexicographic order, so that $v^0=\varnothing$. The \textit{Lukaciewicz
path} $V$ of $\theta$ (sometimes known as the depth-first path) is the
continuous function $(V_t = V^\theta_t, t
\in[0, \# \theta])$ defined as follows: for
$n\in\{1,\ldots,\#\theta\}$
\[
V_n = V^\theta_n:= \sum_{i=0}^{n-1} (k_{v^i}-1 ),
\]
and between integers $V$ is interpolated linearly.\setcounter{footnote}{2}\footnote{In
\cite{LeGall2005,DuqLeG2002},
the Lukaciewicz path is defined as a piecewise constant,
discontinuous function, but there the case when the scaling limit of this
path is discontinuous is also treated. Note that only the values of $V_n,
n \in\{1,\ldots,\#\theta\}$, are needed to recover the tree $\theta$.
Moreover, in our case, $\sup_{t \ge0} \vert V_{t+1}-V_t\vert$ is
bounded by
$\sigma$, so that the eventual scaling limit will be continuous. The
advantage of our convention is that it allows us to consider locally
uniform convergence of the rescaled Lukaciewicz paths in a space of
continuous functions.}

The values $V_n$ are also given by the following \textit{right-hand
description of
the Lukaciewicz path}. This description is simpler to visualize, though
we do not know of a reference for it. For $v \in\theta$, consider the
subtree $\theta^v \subset\theta$ formed by all the vertices which are
smaller or equal to $v$ in the lexicographic order. Let $n(v,\theta)$ be
the number of edges connecting vertices of $\theta^v$ with vertices of
$\theta\setminus\theta^v$. Then
\[
V(k) = n(v^k,\theta) - 1.
\]
The reason we call this the right-hand description is that $n(v,\theta
)$ is also
the number of edges attached on the right-hand side of the path from
$\varnothing$
to $v$. It is straightforward to check that this description is consistent
with other definitions.

The height function is the second encoding we wish to consider. We also
define it to be a piecewise linear function\footnote{Again, in \cite
{LeGall2005},
the height function of a nondegenerate tree is discontinuous.} with
$H(k)$ the height of $v^k$ above the root. It is related to the Lukaciewicz
path by
%
\begin{equation}\label{VtoH}
H(n) = \# \bigl\{ k<n\dvtx V_k = \min\{ V_k,\ldots,V_n \} \bigr\}.
\end{equation}

Finally, the contour function of $\theta$ is obtained by considering a
walker exploring $\theta$ at constant unit speed, starting from the
root at
time 0, and going from left to right. Each edge is traversed twice
(once on
each side), so that the total time before returning to the root is
$2(\#\theta-1)$. The value $C^\theta(t)$ of the contour function at
time $t
\in[0, 2(\# \theta-1)]$ is the distance between the walker and the root
at time $t$.

It is straightforward to check that the Lukaciewicz path, height function
and contour function each uniquely determine---and hence represent---any
finite tree~$\theta$. Figure~\ref{figencoding} illustrates these
definitions, as
they are easier to understand from a picture.

%
\begin{figure}

\includegraphics{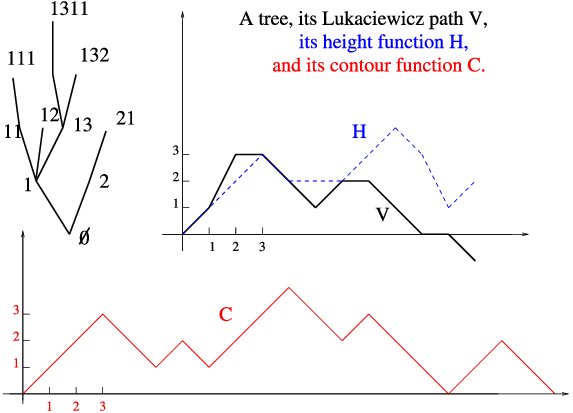}
\caption{A finite tree and its encodings.}
\label{figencoding}\label{figgraphVHC}
\end{figure}

At times it is useful to encode a sequence of finite trees by a single
function. This is done by concatenating the Lukaciewicz paths or height
function of the trees of the sequence. Note that when coding a sequence of
trees, jumping from one tree to the next corresponds to reaching a new
integer infimum in the Lukaciewicz path, while it corresponds to a
visit to 0 in
the height process.

\subsection{Encoding sin-trees} \label{subenc2}

While the definitions of Lukaciewicz path, and height and contour functions,
extend immediately to infinite (discrete) trees, these paths generally no
longer encode a unique infinite tree. For example, all the trees containing
the infinite branch $\{\varnothing, 1, 11, 111,\ldots\}$ would have the
identity function for height function, so that equal paths correspond to
distinct infinite trees. In fact, the only part of an infinite tree which
one can recover from the the height and contour functions is the subtree
that lies left of the leftmost infinite branch. The Lukaciewicz path
encodes additionally the degrees of vertices along the leftmost infinite
branch.

However, if we restrict the encodings to the class of trees whose only
infinite branch is the rightmost branch, then the three encodings still
correspond to unique trees. In particular, observe that $\IPC_G$ and
$\cR$ are fully encoded by their Lukaciewicz paths (as well as by their
height, or contour functions). That is the reason we begin our discussion
with these conditioned objects.

Not surprisingly, it is possible to encode any sin-tree, such as the IIC
and IPC, by using \textit{two} coding paths, one for the part of the tree
lying to the left of the backbone, and one for the part lying to its right.
More precisely, suppose $\cT$ is a sin-tree, and $\BB$ denotes its
backbone. The left tree is defined as the set of
all vertices on or to the left of the backbone:
\[
\cT_{G}:= \bigcup_{v\in\mathrm{BB}} \cT^v = \{ x \in\cT\dvtx \exists
v \in
\BB,
x\le v\}.
\]
We do not define the right-tree of $\cT$ as the set of vertices which lie
on or to the right of the backbone. Rather, in light of the way the
encodings are defined, it is easier to work with the mirror-image of
$\cT$,
denoted $\overline{\cT}$ and defined as follows: since a plane tree
is a
tree where the children of each vertex are ordered, $\overline{\cT}$
may be
defined as the same tree but with the reverse order on the children at each
vertex. We then define
\[
\cT_D = (\overline{\cT})_G.
\]

Obviously, only the rightmost branches of $\cT_G, \cT_D$ are
infinite, so
the Lukaciewicz paths $V_G,V_D$, of $\cT_G,\cT_D$, do encode uniquely each
of these two trees (and so do the height functions $H_G, H_D$ and the
contour functions $C_G, C_G$). Therefore, the pair of paths $(V_G,V_D)$
encodes $\cT$ [and so do the pairs $(H_G,H_D)$, $(C_G,C_D)$]. Note that
$H_G, C_G$ are also, respectively, the height and contour functions of
$\cT$
itself, while $H_D,C_D$ are, respectively, the height and contour functions
of $\overline{\cT}$.

\subsection{Overview} \label{subheuristics}

Let us try to give briefly, and heuristically, some intuition of why
Theorem~\ref{Trightscale} holds. For $t>0$, the tree emerging from
$\BB_{[kt]}$
is coded by the $[kt]$th excursion of $V$ above $0$. Except for its first
step, this excursion has the same transition probabilities as a random walk
with drift $\sigma\widehat{W}_{[kt]}-1$, which, by the\vadjust{\goodbreak}
convergence\vspace*{1pt}
(\ref{eWtoL}), is approximately $-L(t)/k$. Additionally, by~\cite{AGdHS2008}, Proposition~3.1,
$\widehat W_n$ is constant for long stretches of
time. It is well known (see, e.g.,~\cite{Jacod1985}, Theorem~2.2.1)
that
a sequence of random walks with drift $c/k$, suitably scaled, converges as
$k \to\infty$ to a $c$-drifted Brownian motion. Thus, we expect to find
segments of drifted Brownian paths in our limit. According to the
convergence (\ref{eWtoL}), the drift is expressed in terms of the
$L$-process. This is what the definition of $Y$ expresses.

Thus, the idea when dealing with either the conditioned or the
unconditioned IPC is to cut these sin-trees into pieces corresponding
to stretches where $\widehat W$ is constant, and to look
separately at the convergence of each piece. Since we deal extensively
with codings of trees by paths, we call these pieces of trees
\textit{segments}, although in the terminology of \cite
{NewmanStein1995,GoodmanPondsInPrep,DamronSapozhnikov2010} and other
works they are known as the \textit{ponds} of the IPC.

In Section~\ref{subunique} we establish existence and uniqueness
results for equation~(\ref{eqSDE}).

In Section~\ref{secfixed} we look at the convergence of the rescaled
paths coding a sequence of such segments for well chosen, fixed values of
the $\widehat W$-process. In fact, we consider slightly more general
settings which
allows us to treat the case of the IIC as well as the various flavors of
the IPC.

In Section~\ref{secproofmain} we prove Theorem~\ref{Trightscale} and
Theorem~\ref{Ttwosidedscale} by combining
segments. 
To deal with the fact that $\widehat W$ is random and
exploit the convergence (\ref{eWtoL}), we use a coupling argument (see
Section~\ref{subcoupling}). We then prove that the segments fall into
the family dealt with in Section~\ref{secfixed}.
Because of the divergence of the $L$-process at the origin, we only
perform the above for subtrees above certain levels, and
bound the resulting error separately. The proof of Theorem~\ref{Tlimexist}
follows from Theorem~\ref{Trightscale}.

Finally, in Section~\ref{seclevels} we apply our convergence results
to establish asymptotics for level and volume estimates of the IPC, to
recover and extend results of~\cite{AGdHS2008}.

\section{\texorpdfstring{Solving (\protect\ref{eqSDE})}{Solving (E(L))}} \label{subunique}

\begin{claim}\label{cl31}
Solutions to (\ref{eqSDE}), $\cE(L/2)$ are unique in law.
\end{claim}

Curiously, we were unable to determine whether the solutions to (\ref
{eqSDE}) are a.s.
pathwise unique (i.e., whether strong uniqueness holds). For our
purposes uniqueness in
law suffices.
\begin{pf*}{Proof of Claim~\ref{cl31}}
We prove this claim for equation (\ref{eqSDE}). The proof for equation
$\cE(L/2)$ is identical.

Let $Y$ be a solution of (\ref{eqSDE}). Since $L$ is positive, $Y_t \le
B_t$. Since $L$ is nonincreasing, $\int_{0}^t L(-\underline{Y}_{
s})\,ds \le
\int_0^t L(-\underline{B}_{ s}) \,ds$. For any fixed $\eps>0$, a.s.
for all small enough $s$, $-\underline{B}_{ s} > s^{1/2-\eps}$, while
a.s. for all small enough $u$, $L(u) < u^{-(1+\eps)}$. We deduce that
almost surely $\lim_{t \to0} \int_{0}^t L(-\underline{Y}_{ s})\,ds =
0$. Thus,
any solution of (\ref{eqSDE}) is continuous.\vadjust{\goodbreak}

Let us now consider two solutions $Y^1$, $Y^2$ of
$\mathcal{E}(L)$ and fix $\eps>0$. Introduce
\[
j^{\eps} := \inf\{ t >0\dvtx L(t) < \eps^{-1}\}
\]
and
\begin{eqnarray*}
t_0^{\eps} &:=& \inf\{t>0\dvtx -\underline{B}_{ t} >
j^{\eps}\},\\
t_1^{\eps} &:=& \inf\{t>0\dvtx -\underline{Y}^1_{ t} >
j^{\eps}\},\\
t_2^{\eps} &:=& \inf\{t>0\dvtx -\underline{Y}^2_{ t} >
j^{\eps}\}.
\end{eqnarray*}

From the continuity of $Y^1, Y^2$ we have
$Y^1(t_{1}^{\eps}) =
Y^2(t_2^{\eps})=-j^{\eps}$. Moreover, we have
a.s. $t_1^{\eps} \vee
t_2^{\eps} \le t_0^{\eps}$, and, therefore,
%
\begin{equation} \label{eqmaxt1t2}
t_1^{\eps} \vee
t_2^{\eps} \xrightarrow{\eps\to0}^{\mathrm{a.s.}} 0.
\end{equation}

Introduce a Brownian motion $\beta$ independent of $B$ and consider
the (SDE)
{\renewcommand{\theequation}{$\cE(\eps,L)$}
\begin{equation}\label{eqSDEeps}
Z_t^{\eps} = \beta_t - \int_0^t
L( j^{\eps}-\underline{Z}^{\eps}_s ) \,ds.
\end{equation}}

\noindent Pathwise existence and uniqueness hold for (\ref{eqSDEeps}) by
standard arguments.

We then define
\begin{eqnarray*}
Y^{1,\eps}_t &=& \cases{
Y^1_t, &\quad if $t < t_1^{\eps}$, \cr
Y^1_{t_{1}^{\eps}} + Z_t^\eps, &\quad if $t \ge t_1^{\eps}$,}
\\
Y^{2,\eps}_t &=& \cases{
Y^2_t, &\quad if $t < t_2^{\eps}$, \cr
Y^2_{t_{2}^{\eps}} + Z_t^\eps, &\quad if $t \ge t_2^{\eps}$.}
\end{eqnarray*}
Clearly, $Y^{1,\eps}, Y^{2,\eps}$ are a.s. continuous, and, moreover,
$Y^1$ and $Y^{1,\eps}$ have the same distribution, and so do $Y^2$ and
$Y^{2,\eps}$. However, $(Y^{i,\eps}(t_i^{\eps}+t))_{t\ge0}$ for
$i=1,2$ have a.s. the same path. From this fact, the continuity of
$Y^{1,\eps}, Y^{2,\eps}$ and (\ref{eqmaxt1t2}), it follows that for any
$F \in\cC_b(\cC(\rr_+,\rr),\rr)$
\[
\vert E[F(Y^{1})] -E[F(Y^{2})]\vert
= \vert E[F(Y^{1,\eps})] -E[F(Y^{2,\eps})
]\vert
\]
goes to 0 as $\eps$ goes to 0, which completes the proof.
\end{pf*}

\section{Scaling simple sin-trees and their segments}
\label{secfixed}

The goal of this section is to establish the convergence of the
rescaled paths encoding suitable sequences of well-chosen segments. In
order to cover the separate cases at once, we will work in a slightly
more general context than might seem necessary. We first look at a
sequence of particular sin-trees $\mathbf{T}^k$ for which the vertices
adjacent to the backbone generate i.i.d. subcritical (or critical)
Galton--Watson trees. The law of such a tree is determined by the
branching law on these Galton--Watson trees and the degrees along the
backbone. If the degrees along the backbone do not behave too
erratically and the percolation\vspace*{1pt} parameter scales
correctly, then the sequence of Lukaciewicz paths $\bV^k$ has a scaling
limit.\vadjust{\goodbreak}

The results for the IIC follow directly.
Also, we determine the scaling limits of the paths encoding a sequence
of subtrees obtained by truncations at suitably vertices on the backbones
of $\mathbf{T}^k$.
These will be important intermediate results in the proofs of
Theorems~\ref{Trightscale} and~\ref{Ttwosidedscale}.

\subsection{Notation}\label{subnotationsegments}

Throughout this section we fix for each $k\in\Z_+$ a parameter $w_k
\in
[0,1/\sigma]$, and denote by $(\theta_n^k)_{n \in\zz_+}$ a sequence of
i.i.d. subcritical Galton--Watson trees with branching law $\Bin
(\sigma,
w_k)$. For each $k$ we also let $Z_k$ be a sequence of random variables
$(Z_{k,n})_{n\ge0}$ taking values in $\Z_+$.
%
\begin{defn}\label{defZktrees}
The \textit{$(Z_k,\theta^k)$-tree} is the sin-tree defined as follows. The
backbone $\BB$ is the rightmost branch. The vertex $\BB_i$ has $1+Z_{k,i}$
children, including $\BB_{i+1}$. Let $v_0,\ldots$ be all vertices adjacent
to the backbone, in lexicographic order, and identify $v_n$ with the root
of the tree $\theta^k_n$.
\end{defn}

Thus, the first $Z_{k,0}$ of the $\theta$'s are attached to children of
$\BB_0$, the next $Z_{k,1}$ to children of $\BB_1$ and so on.
We will use the notation $\bT^k$ to designate the
$(Z_k,\theta^k)$-tree, and $\bV^k$ for its Lukaciewicz path.
%
\begin{defn}\label{itruncation}
Let $T$ be a sin-tree whose backbone is its rightmost branch. For $i
\in
\Z_+$, let $\BB_i$ be the vertex at height $i$ on the backbone of $T$.
The \textit{$i$-truncation} of $T$ is the subtree
\[
T^i:= \{ v \in T\dvtx v \le\BB_{i}\},
\]
where $\le$ denotes lexicographic ordering.
\end{defn}

Thus, the $i$-truncation of a tree consists of the backbone up to $\BB_i$
and the subtrees attached strictly below level $i$. We denote by
$\bT^{k,i}$ the $i$-truncation of $\bT^k$, and by $\bV^{k,i}$ its
Lukaciewicz path. We further define $\tau^{(i)}$ as the time of the
$(i+1)$th return to 0 of $\bV^k$; here we suppress the dependence of
$\tau^{(i)}$ on $k$. Observe then that $\bV^{k,i}$ coincides
with $\bV^{k}$ up to the time $\tau^{(i)}$, takes the value $-1$ at
$\tau^{(i)}+1$, and terminates at that time.

It will be useful to study first the special case where $Z_k$ is a sequence
of i.i.d. binomial $\Bin(\sigma,w_k)$ random variables. Observe that in
this case the subtrees attached to the backbone are i.i.d. Galton--Watson
trees [with branching law $\Bin(\sigma,w_k)$]. We use calligraphed letters
for the various objects in this case. We denote the binomial variables
$\cZ_{k,n}$, we write $\cT^k$ for the corresponding $(\cZ_k,
\theta^k)$-tree, $\cT^{k,i}$ for its $i$-truncation, and $\cV^k, \cV^{k,i}$
for the corresponding Lukaciewicz paths.

In the perspective of proving our main results, we note another special
distribution of the variables $Z_{k,n}$ that is of interest. If
$Z_{k,n}$ are i.i.d. $\Bin(\sigma-1, w_k)$, then the subtrees
emerging from the backbone of the $(Z_k,\theta^k)$-tree are independent
subcritical percolation clusters with parameter $w_k$. In particular, for
suitably chosen values of $w_k, n_k$, $\bT^{k,n_k}$ has the same law
as a\vadjust{\goodbreak}
certain segment of $\cR$. On the other hand, if $w_k\equiv\sigma
^{-1}$, then the
corresponding $(Z,\theta)$-tree is simply the IIC conditioned on its
backbone being the rightmost branch of $\mathcal{T}$, which we denote
by $\IIC_{\cR}$. We will see
below
that the IIC with unconditioned backbone, as well as segments of the
unconditioned IPC,
can be treated in a similar way.


\subsection{Scaling of segments}
\label{subscalesegments}

\begin{prop}\label{prfixeddrift}
Let $Z_{k,n}$ be random variables satisfying the following assumptions:
%
{\renewcommand{\theequation}{$\cA$}
\begin{equation}\label{eqBBchildrenAssumps}
\cases{
\mbox{For any $k$, the variables $(Z_{k,n})_{n}$ are i.i.d.;} \cr
\mbox{for some $C,\alpha>0$, $\E Z_{k,n}^{1+\alpha} < C$ for any $k$;}
\cr
\mbox{for some $\eta>0$, $\P(Z_{k,n}>0) > \eta$ for any $k$;} \cr
\mbox{if $m_k=\E Z_{k,n}$ then $m=\lim m_k$ exists.}}
\end{equation}}

\noindent
Further assume that
$w_k\leq\sigma^{-1}$ satisfy $\lim_k k(1-\sigma w_k)=u$. Then, as $k
\to
\infty$, weakly in $\cC(\R_+,\R)$,
%
\setcounter{equation}{22}
\begin{equation} \label{convexc}
\biggl(\frac{1}{k}\bV^k_{[k^2 t]}\biggr)_{t\ge0}
\xrightarrow{k\to\infty} (X_t)_{t\ge0},
\end{equation}
where $X_t = Y_t-\underline{Y}_t$ and $Y_t = B_{\gamma t} - u t$ is a
drifted Brownian motion.
\end{prop}

Since our goal is to represent segments of the IPC as well-chosen $\bT
^{k,i}$, we have to deduce from
Proposition~\ref{prfixeddrift} some results for the coding paths of
the truncated
trees. The convergence will take place in the space of continuous stopped
paths denoted $\cS$. An element $f \in\cS$ is given by a lifetime
$\zeta(f)\geq0$ and a continuous function $f$ on $[0,\zeta(f)]$.
$\cS$ is a Polish space with metric
\[
d(f,g) = \vert\zeta(f)-\zeta(g)\vert+\sup_{t\leq\zeta(f)\wedge
\zeta
(g)} \{
\vert f(t)-g(t)\vert \}
.
\]

It is clear from the right-hand description of Lukaciewicz paths that
the path
of $\bT^{k,i}$ visits 0 exactly when reaching backbone vertices. In
particular, its length is~$\tau^{(i)}$, the time of the
$i$th return to 0 by the path $\bV^k$. We shall use this to prove the
following.
%
\begin{corollary}\label{cornkexc}
Assume the conditions of Proposition~\ref{prfixeddrift} are in force.
Assume further that $0<x=\lim n_k/k$. Then, weakly in $\cS$,
%
\begin{equation}\label{convnkexc}
\biggl(\frac{1}{k} \bV_{[k^2 t]}^{k,n_k}\biggr)_{t \le\tau^{(n_k)}/k^2}
\xrightarrow{k\to\infty} (X_t)_{t\le\tau_{m x}},
\end{equation}
where $X$ and $Y$ are as in Proposition~\ref{prfixeddrift}, and
$\tau_y$ is the stopping time $\inf\{t>0\dvtx Y_t=-y \}$.
\end{corollary}

It is then straightforward to deduce convergence of the height functions.
Let $h^{k}$ (resp., $h^{k,i}$) denote the height function coding the tree
$\bT^k$ (resp.,~$\bT^{k,i}$).\vadjust{\goodbreak}

\begin{corollary}\label{corheight,fixed}
Suppose the assumptions of Corollary~\ref{cornkexc} are in force. Then
weakly in $\cC(\R_+,\R)$,
%
\begin{equation}\label{eqheightfull}
\biggl(\frac{1}{k} h^{k}_{[tk^2]}\biggr)_{t\ge0} \xrightarrow{k\to\infty}
\biggl(\frac{2}{\gamma}(Y_t - \underline{Y}_{ t}) - \frac
{1}{m}\underline{Y}_{ t}\biggr)_{t\ge0}.
\end{equation}
Furthermore, weakly in $\cS$,
%
\begin{equation}\label{eqheightnk}
\biggl(\frac{1}{k} h^{k,n_k}_{[tk^2]}\biggr)_{t\le\tau^{(n_k)}/k^2}
\xrightarrow{k\to\infty}
\biggl(\frac{2}{\gamma} (Y_t - \underline{Y}_{ t}) - \frac
{1}{m} \underline{Y}_{ t}\biggr)_{t \le\tau_{m x}}.
\end{equation}
\end{corollary}

\subsection{\texorpdfstring{Proof of Proposition \protect\ref{prfixeddrift}}{Proof of Proposition 4.3}}
\label{secfixeddriftproof}

We start with the following lemma, which relates the Lukaciewicz paths
of a
sequence of trees, and that of the tree consisting of a backbone to
which the trees of the sequence are attached.

\begin{lemma}\label{LgluetoBB}
Let $(\theta_n)_{n\geq0}$ be a sequence of trees, and define the sin-tree
$T$ to be the sin-tree with a backbone $\BB$ on the right, such that the
root of $\theta_n$ is identified with $\BB_n$. Let $U$ be the Lukaciewicz
path coding the sequence $\theta$, and let $V$ be the Lukaciewicz path of
$T$. Then
\[
V_n = U_n + 1 -\underline{U}_{n-1},
\]
where $\underline{U}$ is the infimum process of $U$ and by convention
$\underline{U}_{-1}=1$.
\end{lemma}
\begin{pf}
The lemma follows directly from the definition of Lukaciewicz paths. $U$
reaches a new infimum (and $\underline{U}$ decreases) exactly when the
process completes the exploration of a tree in the sequence. The
increments of $V$ differ from the increments of $U$ only at vertices of
the backbone of $T$, where the degree in $T$ is one more than the degree
in $\theta_n$.
\end{pf}

We first establish the proposition in the special case introduced earlier,
where $\cZ_k$ is a sequence of i.i.d. $\Bin(\sigma,w_k)$ random variables.
In this case, the subtrees attached to the backbone of $\cT^k$ are a
sequence of i.i.d. Galton--Watson trees with branching law having expectation
$\sigma w_k$ (which tends to 1 as $k\to\infty$) and variance $\sigma w_k
(1-w_k)$ (which tends to $\gamma$ as $k\to\infty$).

The Lukaciewicz path $\cU^{k}$ of this sequence of Galton--Watson trees
is a
random walk with drift $\sigma w_k-1$ and stepwise variance $\sigma w_k
(1-w_k)$. From a well-known extension of Donsker's invariance principle
(see, e.g.,
\cite{Jacod1985}, Theorem II.3.5),
it follows that
\[
\biggl(\frac{1}{k} \cU^k(k^2 t)\biggr)_{t\ge0} \xrightarrow{k\to\infty}
(Y_t)_{t\ge0}
\]
weakly in $\cC(\R_+,\R)$. It now follows from Lemma~\ref{LgluetoBB} that
%
\begin{equation}\label{eqVlimit}
\biggl(\frac{1}{k} \cV^k(k^2 t)\biggr)_{t\ge0} \xrightarrow{k\to\infty}
(X_t)_{t\ge0}.
\end{equation}

Having Proposition~\ref{prfixeddrift} for $\cZ_{k,n}$, we now
extend it
to other degree sequences. By the Skorokhod representation theorem, we may
assume (by changing the probability space as needed) that
(\ref{eqVlimit}) holds a.s.:
%
\begin{equation}\label{eqconvforGW}
\biggl(\frac{1}{k} \cV^k(k^2 t)\biggr)_{t\ge0} \xrightarrow{k \to\infty
}^{\mathrm{a.s.}}
(X_t)_{t\ge0}.
\end{equation}
%
We further couple the trees $\cT^k$ and $\bT^k$ (on a suitable
probability space where the sequences $Z_k$ are defined) by using the
same sequences $\theta^k$ of off-backbone trees. Namely, the subtree
descended from the $n$th vertex adjacent to the backbone, in
lexicographic order, is $\theta_n^k$ for both $\cT^k$ and $\bT^k$, and
we will identify $v\in\theta_n^k$ with the corresponding vertices of
$\cT^k$ and $\bT^k$. However, because the sequences $Z_k$ and $\cZ_k$
are different, the Lukaciewicz paths of these two trees differ, and we
now give bounds to control this difference.

It will be convenient to consider the sets of points
\[
\bG^k:= \{(i,\bV^{k}(i)), i \in\Z_+\},\qquad
\cG^k:= \{(i,\cV^{k}(i)), i \in\Z_+\},
\]
which are the integer points in the graphs of $\bV^k,\cV^k$. To each vertex
$v \in\bT^k$ corresponds a point $(\bx_v, \by_v)\in\bG^k$ [and
similarly $(x_v,y_v)\in\cG^k$ for $v\in\cT^k$].
From the right-hand description of Lukaciewicz paths introduced in Section~\ref{subenc1}, we see that
\begin{eqnarray*}
\bG^k &=& \{(\bx_v,\by_v)\dvtx v \in\bT^k\} = \bigl\{
\bigl( \#(\bT^k)^v, n(v,\bT^k)-1\bigr)\dvtx v \in\bT^k \bigr\}, \\
\cG^k &=& \{(x_v,y_v)\dvtx v \in\cT^k\} = \bigl\{
\bigl( \#(\cT^k)^v, n(v,\cT^k)-1\bigr)\dvtx v \in\cT^k \bigr\}.
\end{eqnarray*}
The next step is to show that these two sets are close to each other. Any
$v\in\theta^k_n$ is contained in both $\bT^k$ and $\cT^k$. We first show
that $\bx_v\approx x_v$ and $\by_v\approx y_v$ for such $v$, and then show
how to deal with the backbones.

Any tree $\theta^k_n$ is attached by an edge to some vertex in the backbone
of $\bT^k$ and $\cT^k$. For any vertex $v\in\theta^k_n$ we denote the
height of this vertex by $\bl_v$ and $\ell_v$, respectively:
\[
\bl_v = \sup\{t\dvtx \BB_t < v \mbox{ in } \bT^k \},\qquad
\ell_v = \sup\{t\dvtx \BB_t < v \mbox{ in } \cT^k \}.
\]
These values depend implicitly on $k$. Note that $\bl_v,\ell_v$ do not
depend on which $v\in\theta^k_n$ is chosen, hence, by a slight abuse of
notation, we also use $\bl_n,\ell_n$ for the same values whenever
$v\in\theta^k_n$.
%
\begin{lemma}
Assume $v\in\theta^k_n$. Then
\begin{eqnarray*}
\vert\bx_v-x_v\vert &=& \vert\bl_v-\ell_v\vert, \\
\vert\by_v-y_v\vert &\leq&\sigma+ Z_{k,\bl_v}.
\end{eqnarray*}
\end{lemma}
\begin{pf}
We have
\[
x_v = \#(\cT^k)^v = \sum_{i<n} \# \theta_i^k + \#(\theta^k_n)^v +
\ell_n\vadjust{\goodbreak}
\]
and, similarly,
\[
\bx_v = \#(\bT^k)^v = \sum_{i<n} \# \theta_i^k + \#(\theta^k_n)^v
+ \bl_n.
\]
The first claim follows.

For the second bound use $\by_v = n(v,\bT^k) -1$. There are
$n(v,\theta^k_n)$ edges connecting $(\bT^k)^v$ to its\vspace*{1pt} complement inside
$\theta^k_n$ and at most $Z_{k,\bl_n}$ edges connecting $\BB_{\bl
_n}$ to
the complement. Similarly, in $\cT^k$ we have the same $n(v,\theta^k_n)$
edges inside $\theta^k_n$ and at most $\cZ_{k,\ell_n}\le\sigma$ edges
connecting $\BB_{\ell_n}$ to the complement. It follows that the
difference is at most $\sigma+Z_{k,\bl_n}$.
\end{pf}

Next we prepare to deal with the backbone. For a vertex $v\in\bT^k$, define
$u\in\bT^k$ by
\[
u = \min\{ u\in(\bT^k\setminus\BB)\dvtx u\ge v \}.
\]
If $v\notin\BB$, then $u=v$. If $v$ is on the backbone, then $u$ is
the first
child of $v$, unless $v$ has no children outside the backbone. Note that
$u\in\theta^k_n$ for some $n$, so we may also consider $u$ as a
vertex of
$\cT^k$. Note also that $v\to u$ is a nondecreasing map from $\bT^k$ to
$\cT^k$.
%
\begin{lemma}\label{LBBdisp}
For a backbone vertex $v$ in $\bT^k$, define $n$ by $\theta^k_n < v <
\theta^k_{n+1}$. Then
\begin{eqnarray*}
\vert\bx_v - \bx_u\vert &\le& 1 + \bl_{n+1} - \bl_n, \\
\vert\by_v - \by_u\vert &\le& \sigma+ Z_{k,\bl_{n+1}}.
\end{eqnarray*}
\end{lemma}
\begin{pf}
The only vertices between $v$ and $u$ in the lexicographic order are $u$
and some of the backbone vertices with indices from $\bl_n$ to
$\bl_{n+1}$, yielding the first bound.

Let $w\in\BB$ be $u$'s parent. If $v$ has children apart from the next
backbone vertex, then $w=v$ and $u$ is $v$'s first child, so $\by
_u-\by_v
= k_u-1 \le\sigma-1$. If $v$ has no other children, then $\by_u-\by
_v =
(k_u-1) + (k_w-1) \le\sigma+ Z_{k,\bl_{n+1}}$.
\end{pf}
%
\begin{lemma}\label{Lhdisp}
Fix $\eps,A>0$ and let $w$ be the $[Ak^2]$th vertex of $\bT^k$. Then with
high probability $\ell_w,\bl_w \leq k^{1+\eps}$.
\end{lemma}
\begin{pf}
Since each $\theta^k_n$ is (slightly) subcritical, we have
$\P(\#\theta^k_n > k^2) > c_1 k^{-1}$ for some $c_1>0$. Consider the first
$k^{1+\eps}$ vertices along the backbone in $\bT^k$. With high
probability, the number of $\theta$'s attached to them is at least
$\eta k^{1+\eps}/2$. On this event, with high probability, at
least $c_2 k^\eps$ of these have size at least $k^2$, hence, there are
$c_2 k^{2+\eps} \gg Ak^2$ vertices $v$ with $\bl_v \le k^{1+\eps}$
(and these
include the first $Ak^2$ vertices in the tree). $\ell_w$ is dealt with
in the same way.
\end{pf}
%
\begin{lemma}\label{Lhvdisp}
Fix $A>0$ and let $w$ be the $[Ak^2]$th vertex of $\bT^k$. For $\eps>0$
small enough,
\[
\P\Bigl( \sup_{v<w} \vert\bx_v - x_u\vert > 3k^{1+\eps} \Bigr)
\xrightarrow{k\to\infty} 0
\]
and
\[
\P\Bigl( \max_{v<w} \vert\by_v-y_u\vert > k^{1-\eps} \Bigr)
\xrightarrow{k\to\infty} 0.
\]
\end{lemma}
\begin{pf}
For a vertex $v\in\theta^k_n$ off the backbone we have $u=v$ and
\[
\vert\bx_v - x_u\vert \le\vert\bl_v - \ell_v\vert \le\bl_v +
\ell_v
\le\bl_w
+ \ell_w,
\]
and with high probability this is at most $2k^{1+\eps}$. If $v<w$ is in
the backbone, then we argue that $\vert\bx_v-\bx_u\vert \ll
k^{1+\eps}$.
To this
end, note that $\bl_{n+1}-\bl_n$ is dominated by a geometric random
variable with mean $1/\eta$ (since the $Z_{k,n}$'s are independent).
Since only $n<Ak^2$ might be relevant to the initial part of the tree,
this shows that with high probability $\vert\bx_v-\bx_u\vert <
c\log k
\ll
k^{1+\eps}$.

The bound on the $y$'s follows from the bounds on $\vert\by
_v-y_u\vert$. All
that is needed is to show that with high probability $Z_{k,n} <
k^{1-\eps}$ for all $n<k^{1+\eps}$, and this follows from assumption
(\ref{eqBBchildrenAssumps}) 
and Markov's inequality.
\end{pf}

We now finish the proof of Proposition~\ref{prfixeddrift}. Because the
path of $\bV^k$ is linearly interpolated between consecutive integers, and
since for any $A>0$ the paths of $X$ are a.s. uniformly continuous on
$[0,A]$, the proposition will follow if we establish that for any
$A,\eps>0$,
%
\begin{equation}\label{eqendprooffd}
\pp\biggl(\sup_{t \in[0,A]} \biggl\vert\frac{1}{k} V^k_{[k^2 t]} -
X_t\biggr\vert
>\eps
\biggr) \xrightarrow{k\to\infty} 0.
\end{equation}

Consider first $t$ such that $k^2t\in\Z_+$. Then there is some vertex
$v\in\bT^k$ so that $\bx_v=k^2t$. Let $u\in\cT^k$ be as defined
above, and
suppose $k^2s=x_u$. Then (\ref{eqconvforGW}) implies that
$\vert k^{-1}y_u-X_s\vert$ is uniformly small. Lemma~\ref{Lhvdisp} implies
that with
high probability $\vert k^2s-k^2t\vert=\vert x_u-\bx_v\vert\le
3k^{1+\eps}$ for
all such $v$.
Thus, $\vert s-t\vert\le k^{-1+\eps} \ll1$. Since paths of $X$ are uniformly
continuous, we find $\vert X_s-X_t\vert$ is uniformly small, and so
$\vert k^{-1}y_u-X_t\vert$
is uniformly small. Finally, Lemma~\ref{Lhvdisp} states that $\vert
y_u-\by _v\vert\leq
C$, so the scaled vertical distance is also~$o(1)$.

Next, assume $m<k^2t<m+1$. Then $\bV^k(k^2t)$ lies between $\bV^k(m)$ and
$\bV^k(m+1)$. Since both of these are close to the corresponding
values of
$X$, and since $X$ is uniformly continuous (and the pertinent points differ
by at most $k^{-2}$), we may interpolate to find that (\ref{eqendprooffd})
holds for all $t<A$.


\subsection{\texorpdfstring{Proofs of Corollaries \protect\ref{cornkexc} and \protect\ref{corheight,fixed}}
{Proofs of Corollaries 4.4 and 4.5}}

\mbox{}

\begin{pf*}{Proof of Corollary~\ref{cornkexc}}
By Proposition~\ref{prfixeddrift}, the limit of the process
$(\smash{\frac{1}{k} V_{[k^2t]}^k})_{t \le\tau^{(n_k)}}$ must
take the form\vspace*{2pt}
$(X_{t})_{t\le\tau}$ for some possibly random time $\tau$, and,
furthermore, $X_{\tau}=0$. We need to show that $\tau= \tau_{mx} =
\inf
\{ t \ge0\dvtx -Y_t = mx \}$.

In the special case of the tree $\cT^k$ we note that the infimum process
$\underline{\cU}^k$ records the index of the last visited vertex
along the backbone. Therefore, $\tau^{(n_k)}$ is the time at which
$\cU^k$ first reaches $-n_k$, and by assumption $n_k\sim xk$.
Using the a.s. convergence of $\frac{1}{k} \cU^k([k^2t])$ toward $Y_t$,
along with the fact that for any fixed $x>0$, $\eps>0$, one has a.s.
$\underline{Y}_{\tau_x -\eps}> -x > \underline{Y}_{\tau_x +\eps
}$, we
deduce that a.s., $\tau^{(n_k)}/k^2 \to\tau_x$. It then follows that
\[
\biggl( \frac{1}{k} \mathcal{V}^{k}_{[k^2 t]}, t \le
\bigl(\tau^{(n_k)}+1\bigr)/k^2 \biggr) \xrightarrow{k \to\infty}^{\mathrm
{a.s.}} (
X_t, t \le\tau_{x} ).
\]
Since, in this case, $m_k = \sigma w_k \to m=1$, this implies the
corollary for this special distribution.

The general case then follows as a consequence of excursion theory. Indeed,
$(-\underline{Y}_{ t}, t \ge0)$ can be chosen to be the local time
at its
infimum of $Y$ (see, e.g.,~\cite{RogWill1994v2}, Paragraph VI.8.55),
that
is, a local time at $0$ of $X$, since excursions of $Y$ away from its
infimum match those of $X$ away from $0$. However, if $N_t^{(\veps)}$
denotes the number of excursions of $X$ away from $0$ that are completed
before $t$ and reach level $\veps$, then $(\lim_{\veps\to0} \veps
N_t^{(\veps)}, t\ge0)$ is also a local time at $0$ of $X$, which
means that it has to be proportional to $(-\underline{Y}_{ t}, t \ge
0)$ (cf.,
e.g.,~\cite{Blumenthal1992}, Section III.3(c) and Theorem VI.2.1).
In other words, there exists a constant $c>0$ such that for any $t \ge0$,
\[
\lim_{\veps\to0} \veps N_t^{(\veps)} = -c \underline{Y}_{ t}.
\]
In the special case when $\mathcal{Z}_{k,n} = \Bin(\sigma, w_k)$ we have
already proven the corollary. In particular, the number
$\mathcal{N}^{k,(\veps)}$ of excursions of $(\frac{1}{k}\cU
_{k^2t}^{k}, t
\le\tau^{(n_k)})$ which reach level $\veps$ is such that, when letting
$k \to\infty$ and then $\veps\to0$, we have $ \veps\mathcal
{N}^{k,(\veps)} \to c
x $.

Let\vspace*{1pt} $N^{k,(\veps)}$ be the number of excursions of $(\frac{1}{k}
\smash{V_{[k^2t]}^{k}}, t \le\tau^{(n_k)})$ which reach level~$\veps$.
It follows from Proposition~\ref{prfixeddrift} that, in distribution,
$N^{k,(\veps)} \to N_{\tau}^{\veps}$ as $k \to\infty$.

However, by assumption $\mathcal{A}$ we can use the law of large
numbers for the sequences $(Z_{k,n})_{n \in\nn}$ along with the fact
that $m_k \to m$, to ensure that $ \veps N^{k,(\veps)} \underset{k
\to
\infty}{\sim} m \veps
\mathcal{N}^{k,(\veps)}$. Therefore, letting first $k\to\infty$, then
$\veps\to0$, we find $ \veps N^{k,(\veps)} \to m c x $.

From the fact that $\tau^{(n_k)}$ are
stopping times, we deduce that $\tau$ itself is a stopping time. Since
$X_{\tau}=0$, for any $s>0$, the local time at $0$ of $X$ (i.e.,
$-\underline{Y}$)
increases on the interval $(\tau, \tau+s)$.\vadjust{\goodbreak} It follows that for a certain
real-valued random variable $R$, $\tau= \tau_R = \inf\{ t \ge0\dvtx -Y_t
= R \}$, and we deduce that, in distribution, $R = mx$, that is, $\tau
= \tau_{mx}$
\end{pf*}
\begin{pf*}{Proof of Corollary~\ref{corheight,fixed}}
The relation between the height function and the Lukaciewicz path is well
known; see, for example,~\cite{DuqLeG2002}, Theorem 2.3.1 and equation~(1.7).
Combining with Proposition~\ref{prfixeddrift}, one finds that the
height process of the \textit{sequence of trees} emerging from the backbone
of $\bT^k$ converges when rescaled to the process
\[
\frac{2}{\gamma} (Y_t - \underline{Y}_{ t}).
\]
Moreover, the difference between the height process of $\bT^k$ and that
of the sequence of trees emerging from the backbone of $\bT^k$ is simply
$-\underline{U}^{k}$. As in the proof of Corollary~\ref{cornkexc}, one
has weakly in $\cC(\R_+,\R)$,
\[
\biggl(-\frac{1}{k} \underline{U}^k_{[k^2t]}\biggr)_{t\ge0}
\xrightarrow{k\to\infty} \biggl(-\frac{1}{m} \underline{Y}_{  t}\biggr)_{t
\ge0},
\]
and (\ref{eqheightfull}) follows. The proof of (\ref{eqheightnk}) is
similar.
\end{pf*}

In fact,~\cite{DuqLeG2002}, Corollary 2.5.1,
states the joint convergence of
Lukaciewicz paths, height and contour functions. It is thus easy to deduce
a strengthening of Corollary~\ref{corheight,fixed} to get the joint
convergence.

\subsection{Two-sided trees}
\label{subtwotrees}

The limit appearing in Proposition~\ref{prfixeddrift} retains very
minimal information about the sequence $Z_k$. If two trees (or two
sides of a tree) are constructed as above using independent $\theta$'s
but dependent sequences of $Z$'s, the dependence between two sequences
might\vspace*{1pt} disappear in the scaling limit. For $k \in\Z_+$, let
$w_k \in[0,1/\sigma]$, and denote by $(\theta_n^k)_{n \in\Z_+}$ and
$(\widetilde{\theta}_n^k)_{n \in \zz _+}$ two \textit{independent}
sequences of i.i.d. subcritical Galton--Watson trees with branching law
$\Bin(\sigma, w_k)$. We let $Z_k$, $\widetilde{Z}_k$ be two sequences
of random variables taking values in $\Z_+$ such that the pairs
$(Z_{k,n}, \widetilde{Z}_{k,n})$ are independent for different $n$;
however, we allow $Z_{k,n}$ and $\widetilde{Z}_{k,n}$ to be correlated.

Let\vspace*{1pt} $\bT^k, \widetilde{\bT}^k$ designate, respectively, the
$(Z_k,\theta^k)$-tree, $(\widetilde{Z}_k, \widetilde{\theta
}^k)$-tree as
defined in Section~\ref{subnotationsegments}. Let\vspace*{1pt} $\bV^k$,
respectively,
$\widetilde{\bV}^k$, denote their Lukaciewicz paths. We recall that
$\bT^{k,
n_k}$, $\widetilde{\bT}^{k,n_k}$ are, respectively, the $n_k$-truncation
of $\bT^k$, respectively, $\widetilde{\bT}^k$,\vspace*{1pt} and we denote by $\bV^{k,n_k},
\widetilde{\bV}^{k,n_k}$ their respective Lukaciewicz paths.
%
\begin{prop}\label{Pindsides}
Suppose $w_k\leq\sigma^{-1}$ is such that $u = \lim_{k\to\infty}
k(1-\sigma w_k)$ exists, and assume that both sequences of variables
$Z_{k,n}, \widetilde{Z}_{k,n}$ satisfy
assumption~(\ref{eqBBchildrenAssumps}). Then, as $k \to \infty$, weakly
in $\cC(\R_+,\R^2)$
\[
k^{-1}\bigl(\bV_{[k^2 t]}^{k}, \widetilde{\bV}_{[k^2t]}^k\bigr)_{t\ge0}
\xrightarrow{k\to\infty}
(X_t, \widetilde{X}_t)_{t\ge0},
\]
where the processes $X,\widetilde{X}$ are two \textit{independent} reflected
Brownian motions with drift $-u$ and diffusion coefficient $\gamma$.

Moreover, if $n_k/k \to x>0$, $m_k \to m$, $\widetilde{m}_k \to
\widetilde{m}$ as
$k \to\infty$, we have
\[
k^{-1}\bigl(\bV_{[k^2 t]}^{k,n_k}, \widetilde\bV_{[k^2 t]}^{k,n_k}\bigr)_{t
\le\tau^{(n_k)}/k^2}
\xrightarrow{k\to\infty}
(X_t, \widetilde X_t)_{t\le\tau_{mx}}.
\]
\end{prop}

The proof is almost identical to that of Proposition~\ref{prfixeddrift}.
When the sequences $Z_k,\widetilde Z_k$ are independent with
$\Bin(\sigma,w_k)$ elements the result follows from
Proposition~\ref{prfixeddrift}. For general sequences, the coupling of
Section~\ref{secfixeddriftproof} shows that the sides have the same
joint scaling limit.

\subsection{Scaling the IIC}
\label{secIIC}

At this point we are already in a position to prove the path
convergence results for the IIC, equations (\ref
{eLukaIICRlimit})--(\ref{eContourIIClimit}) from Theorem~\ref
{TIICresults1}.
As discussed in Section~\ref{subIICrecall}, the IIC is the result
of setting $w_k = 1/\sigma$ in the
above constructions. Specifically, let us first suppose that $Z$ is a
sequence of i.i.d. $\Bin(\sigma-1, 1/\sigma)$ variables and
$(\theta
_n)_n$ is a sequence
of i.i.d. $\Bin(\sigma, 1/\sigma)$ Galton--Watson trees. Let $\bT$
be a
$(Z,\theta)$-tree: then $\bT$ has the same distribution as
$\IIC_{\cR}$.

The convergence of the rescaled Lukaciewicz path encoding this sin-tree
to a
time-changed reflected Brownian path is thus a special case of
Proposition~\ref{prfixeddrift}. The scaling limits of the height and
contour functions follow from Corollary~\ref{corheight,fixed}. We have
$m=\gamma$, so both limits are $\frac{2}{\gamma} B_{\gamma t} -
\frac{3}{\gamma} \underline{B}_{\gamma t}$.

For the IIC with unconditioned backbone, let $Y_n$ be i.i.d. uniform
in $\{1,\ldots,\sigma\}$. Let
$Z_n\sim\Bin(Y_n-1,1/\sigma)$ and $\widetilde Z_n\sim
\Bin(\sigma-Y_n,1/\sigma)$, independent conditioned on $Y_n$ and
independently of all other $n$. Moreover, suppose that $\theta,
\widetilde{\theta}$ are two independent sequences of i.i.d. $\Bin
(\sigma,
1/\sigma)$ Galton--Watson trees. Then, $\bT$ and $\widetilde\bT$ are jointly
distributed as $\IIC_G$ and $\IIC_D$.

Since in this case $m=\widetilde m=\gamma/2$, from
Proposition~\ref{Pindsides} we see that the rescaled Lukaciewicz paths
encoding these two trees converge toward a pair of independent
time-changed reflected Brownian motions, and similarly for the
right/left height and contour functions of the IIC. 

The proofs of the remaining parts of Theorems~\ref{TIICresults1} and
\ref{TIICresults2} are identical to the proofs for the IPC, which are
given in the next two sections.

\section{Bottom-up construction} \label{secproofmain}

\subsection{Right grafting and concatenation}

\begin{defn}
Given a finite plane tree, its \textit{rightmost leaf} is the maximal
vertex in the lexicographic order; equivalently, it is the last vertex to
be reached by the contour process, and is the rightmost leaf of the
subtree above the rightmost child of the root.
\end{defn}
%
\begin{defn}
The \textit{right-grafting} of a plane tree $S$ on a finite plane tree~$T$,
denoted $T\oplus S$, is the plane\vadjust{\goodbreak} tree resulting from identifying the root
of $S$ with the rightmost leaf of $T$. More precisely, let $v$ be the
rightmost leaf of $T$. The tree $T \oplus S$ is given by its set of
vertices $\{u\dvtx u \in T \setminus\{v\} \mbox{ or } u=vw, w \in S\}$.
\end{defn}

Note, in particular, that the vertices of $S$ have been relabeled in
$T\oplus S$ through the mapping from $S$ to $T \oplus S$ which maps
$w$ to $vw$.
%
\begin{defn}
The \textit{concatenation} of two functions $V_1,V_2 \in\cS$ with $V_2(0)=0$,
denoted $V=V_1\oplus V_2$, is defined by
\[
V(t) = \cases{ V_1(t), &\quad $t\leq\zeta(V_1)$, \cr
V_1(\zeta(V_1)) + V_2\bigl(t-\zeta(V_1)\bigr), &\quad $t\in
[\zeta(V_1),\zeta(V_1)+\zeta(V_2)]$.}
\]
\end{defn}
%
\begin{lemma}\label{Lconcatcommute}
If each $Y_i\in\cS$ attains its minimum at $\zeta(Y_i)$, then
\[
\bigoplus(Y_i - \underline Y_i) = \bigoplus Y_i - \underline
{\bigoplus
Y_i}.
\]
\end{lemma}

The following is straightforward to check, and may be used as an alternate
definition of right-grafting.
%
\begin{lemma}\label{Lgraft}
Let $R=T\oplus S$ be finite plane trees, and denote the Lukaciewicz path
of $R$ (resp., $T,S$) by $V_R$ (resp., $V_T,V_S$). Let $V'_T$ be $V_T$
terminated at $\#T$ (i.e., without the final value of $-$1). Then
$V_R=V'_T\oplus V_S$.
\end{lemma}

Consider a sin-tree $T$ in which the backbone is the rightmost path
(i.e.,
the path through the rightmost child at each generation). Given some
increasing sequence $\{x_i\}$ of vertices along the backbone, we cut the
tree at these vertices: let
\[
\widetilde{T}_i:= \{v \in T\dvtx x_i \le v \le x_{i+1} \}.
\]
Thus, $\widetilde{T}_i$ contains the segment of the backbone
$[x_i,x_{i+1}]$ as well as all the subtrees connected to any vertex of
this segment except $x_{i+1}$. We let $T_i$ be $\widetilde T_i$
rerooted at~$x_i$ (formally, $T_i$ contains all $v$ with $x_i
v\in\widetilde T_i$). It is clear from the definitions that $T =
\bigoplus_{i=0}^\infty T_i$. Note that apart from being increasing, the
sequence $x_i$ is arbitrary.




\subsection{\texorpdfstring{IPC structure and the coupling: Proof of Theorem \protect\ref{Trightscale}}
{IPC structure and the coupling: Proof of Theorem 1.2}} \label{subcoupling}

In this section we prove Theorem~\ref{Trightscale}.

Recall the $\widehat{W}$-process introduced in Section \ref
{subrecall} and
the convergence (\ref{eWtoL}). The $\widehat{W}$-process is constant
for long
stretches, giving rise to a partition of $\cR$ into what we shall call
segments. Each segment consists of an interval of the backbone along which
$\widehat{W}$ is constant, together with all subtrees attached to the
interval. To be precise, define $x_i$ inductively by $x_0=0$ and
$x_{i+1} =
\inf_{n>x_i} \{\widehat{W}_n>\widehat{W}_{x_i}\}$. With a slight abuse,
we also let
$x_i$ designate the vertex along the backbone at height $x_i$.\vadjust{\goodbreak}


The backbone is the union of the intervals $[x_i,x_{i+1}]$ for all
$i\geq0$, and the rest of the IPC consists of subcritical percolation
clusters attached to each vertex of the backbone $y\in[x_i,x_{i+1})$. We
can now write
\[
\cR= \bigoplus_{i=0}^\infty R_i,
\]
where $R_i$ is the $[x_i,x_{i+1}]$ segment of $\cR$, rerooted at $x_i$.
$R_i$ has a rightmost branch of length $n_i:= x_{i+1}-x_i$. The degrees
along this branch are i.i.d. $\Bin(\sigma-1, \widehat{W}_{x_i})$,
and each
child off the rightmost branch is the root of an independent Galton--Watson
tree with branching law $\Bin(\sigma, \widehat{W}_{x_i})$. In what
follows, we
say that $R_i$ is a $\widehat{W}_{x_i}$-segment of length $n_i$, and
we observe
that these segments fall into the family dealt with in
Section~\ref{secfixed}.

We may summarize the above in the following lemma:
%
\begin{lemma}
Suppose $\widehat{W}$ consists of values $U_i$ repeated $n_i$ times. Then
$R_i$ is distributed as a $U_i$-segment of length $n_i$ and conditioned
on $\{U_i,n_i\}$, the trees $\{R_i\}$ are independent.
\end{lemma}

A difficulty we must deal with is that in the scaling limit there is no
first segment, but rather a doubly infinite sequence of segments.
Furthermore, the initial segments are far from critical, and so need to be
dealt with separately. This is related to the fact that the Poisson lower
envelope process $L(t)$ diverges near 0 and has no ``first segment.''
Because of
this we restrict ourselves at first to a slightly truncated invasion
percolation cluster. For any $\beta>0$ we define
\[
x_0^\beta= \min\{x\dvtx
\sigma\widehat{W}_x > 1 - \beta/k\},\qquad
x_{i+1}^\beta= \min\{\smash{x > x_i^\beta\dvtx \widehat{W}_x >
\widehat {W}_{x_i^\beta}}\}.
\]
Note
that $x^\beta_0=x_m$ for some $m$ and that $x^\beta_i=x_{m+i}$ for
the same $m$ and all~$i$.


Since we have convergence in distribution of the process $\widehat{W}$,
we may couple
the IPCs for different $k$'s so that the convergence holds a.s. (This
means that the random tree $\cR$ depends on $k$; we will leave this
dependence implicit.) More precisely, let $(j_i^\beta)_{i\in\Z}$ be the
sequence of jump times for $\{L(t)\}$, indexed such that $L(j_0^\beta
)<\beta<L(j_{-1}^\beta)$ a.s. [We may do this since a.s. $\beta$ is
not in the range of $L(t)$.] By the convergence (\ref{eWtoL}) and the
Skorohod representation\vspace*{1pt} theorem, we may assume
that a.s. for any $t\notin J$ we have $k^{-1}(1-\sigma
\widehat{W}^k_{[kt]})\xrightarrow{k\to\infty}L(t)$. Indeed, we will
assume further that $k^{-1} x_i^\beta\to j_i^\beta$ a.s. for each $i$.
This slightly stronger statement follows from (\ref{eRateAsymps}),
which shows that $(k(1-\sigma\widehat W_{[kt]}))$ and $L(t)$ have
asymptotically the same total jump rate. In other words, there are no
``small'' jumps of $\widehat W$ that disappear in the scaling limit $L(t)$.


Denote by $V^\beta_i$ (implicitly depending on $k$) the Lukaciewicz path
corresponding to the $i$th segment $R^\beta_i$ in $\cR^\beta$. For any
$\beta,i$, the $i$th segment has associated percolation parameter
$w_i^\beta$
satisfying $k(1-\sigma w_i)\xrightarrow{k\to\infty} L(j_i^\beta)$ and
length $n_i^\beta$ satisfying $k^{-1}n^\beta_i\rightarrow
j_{i+1}^\beta-j_i^\beta$. By Corollary~\ref{cornkexc}, we have the
convergence in
distribution
%
\begin{equation}
\label{eSegmentConvergence}
\bigl( k^{-1} V^\beta_i(k^2 t), 0\le t \le\tau^{(n_i^\beta)}
\bigr)
\xrightarrow{k\to\infty}\bigl( X_t, 0\le t \le\tau_{\gamma
(j_{i+1}^\beta-j_i^\beta)}\bigr),
\end{equation}
where $X_t = Y_t-\underline Y_{ t}$, and $Y_t$ solves
\[
d Y_t = \sqrt{\gamma}   \,dB_t - L(j_i^\beta)  \,dt.
\]
As in the previous section, $\tau^{(n_i^\beta)}$ denotes the lifetime of
$V^\beta_i$ [i.e., its $(n_i^\beta)$th return to $0$] and $\tau_{y}$
is the
hitting time of $-y$ by $Y$.

Because the convergence in (\ref{eSegmentConvergence}) holds for all
$\beta,i\in\N$, we may construct
the coupling of the probability spaces so that the convergence is also
almost sure, and this is the final constraint in our coupling.
%
\begin{lemma} \label{Lbeta}
Fix $\beta>0$. In the coupling described above we have, almost surely,
the scaling limit
\[
k^{-1}\cV^\beta(k^2 t) \xrightarrow{k\to\infty} X_t,
\]
where
$X_t = \cY^\beta_t-\underline\cY^\beta_{ t}$, and $\cY^\beta$ solves
\begin{equation}
\label{eYbetaEquation}
\cY^\beta_t = \sqrt{\gamma} B_t - \int_0^t L\biggl(j^\beta_0 -
\frac{1}{\gamma} \underline\cY^{\beta}_{  s}\biggr) \,ds.
\end{equation}
\end{lemma}
\begin{pf}
Solutions of the equation for $\cY^\beta$ are a concatenation of
segments. In each segment the drift is fixed, and each segment terminates
when $\underline\cY^\beta$ reaches a certain threshold. The
corresponding segments of $X$ exactly correspond to the scaling limit of
the tree segments $R^\beta_i$.

Lemma~\ref{Lbeta} then follows from Lemmas~\ref{Lconcatcommute} and
\ref{Lgraft}.
\end{pf}
%
\begin{lemma} \label{Lbetatoinfty}
Almost surely,
\[
(\cY^\beta_t, t>0) \xrightarrow{\beta\to\infty} \cY_t,
\]
where $\cY$ solves
\[
\cY_t = \sqrt{\gamma} B_t - \int_0^t L\biggl(-\frac{1}{\gamma
}\underline
\cY_{  s}\biggr) \,ds.
\]
\end{lemma}
\begin{pf}
Consider the difference between the solutions for a pair $\beta<\beta'$.
We have the relation
\[
\cY^{\beta'} = Z \oplus\cY^\beta,\vadjust{\goodbreak}
\]
where $Z$\vspace*{-2pt} is a solution of $ Z_t = \sqrt{\gamma} B_t - \int_0^t
L(j_0^{\beta'} -\frac{1}{\gamma}\underline Z_{  s})\,ds $,
killed when $\underline Z$ first reaches $\gamma(j^{\beta'}_0 -
j^\beta_0)$. In particular, $Z$ is a stochastic process with drift in
$[-\beta',-\beta]$ (and quadratic variation $\gamma$). Thus, to show that
$\cY^\beta$ is close to $\cY^{\beta'}$, we need to show that $Z$ is small
both horizontally and vertically, that is, $\zeta(Z)$ is small, as is
$\Vert Z\Vert_\infty$.

The vertical translation of $\cY^\beta$ is $\sqrt{\gamma}
k^{-1}(x^\beta_0-x^{\beta'}_0)$, which is at most $k^{-1}x^\beta_0$. From
\cite{AGdHS2008}
we know that this tends to 0 in probability as
$\beta\to\infty$. This convergence is a.s. since $x^\beta_0$ is
nonincreasing in $\beta$.

The values of $Z$ are unlikely to be large, since $Z$ has a nonpositive
(in fact, negative) drift and is killed when $\underline Z$ reaches some
negative level close to 0.

Finally, there is a horizontal translation of $\cY^\beta$ in the
concatenation. This translation is just the time at which $\underline Z$
first reaches $\gamma(j^{\beta'}_0-j^\beta_0)$, which is also small,
uniformly in $\beta'$.
\end{pf}

Theorem~\ref{Trightscale}(\ref{equ1}) is now a simple consequence of
Lemmas~\ref{Lbeta}
and~\ref{Lbetatoinfty}. Indeed, the process $\cY-\underline{\cY}$
has the
same law as the right-hand side of (\ref{equ1}), due to the scale invariance of
solutions of (\ref{eqSDE}). We shall note that, in fact, $\cY$ is the
limit of
the rescaled Lukaciewicz path coding the sequence of off-backbone trees.

The same argument using Corollary~\ref{corheight,fixed} instead of
Corollary~\ref{cornkexc} gives the convergence of the height function.

Finally, convergence of contour functions is deduced from that of
height functions by a routine argument (see, e.g.,~\cite{LeGall2005},
Section 1.6).

\subsection{\texorpdfstring{The two-sided tree: Proof of Theorem \protect\ref{Ttwosidedscale}}
{The two-sided tree: Proof of Theorem 1.3}} \label{subtwo-sided}

For convenience we use the shorter notation $\cT$ to designate the IPC,
and we recall the left
and right trees $\cT_G$ and $\cT_D$ as introduced in
Section~\ref{subenc2}. The two trees $\cT_G$ and $\cT_D$ obviously
have the same distribution,
but are not independent. As in the previous section, we may cut these two
trees into segments along which the $\widehat{W}$-process is constant. More
precisely,
\[
\cT_G = \bigoplus_{i=0}^{\infty} T_G^i, \qquad \cT_D =
\bigoplus_{i=0}^{\infty} T_D^i,
\]
where the distribution of $T_D^i, T_G^i$ can be made precise as follows.

Let\vspace*{1pt} $(\theta_n^i)_{n}, (\widetilde\theta_n^i)_{n}$ be sequences
of Galton--Watson trees with branching law $\Bin(\sigma,
\widehat{W}_{x_i})$, all independent. Let $Y_n, n \in\Z_+$ be
independent uniform on
$\{1,\ldots,\sigma\}$, and, conditionally on $Y_n$, let $Z_n$ be $\Bin(Y_n-1,
\widehat{W}_{x_i})$ and $\widetilde Z_n$ be $\Bin(\sigma
-Y_n,\widehat
{W}_{x_i})$,
where conditioned on the $Y$'s all are independent. Then $T_G^i$ and
$T_D^i$ are distributed as the $n_i$-truncations of the
$(Z,\theta^i)$-tree, respectively, of the $(\widetilde{Z},\widetilde
{\theta}^i)$-tree
(constructed as in Definition~\ref{defZktrees}).

The rest of the proof of Theorem~\ref{Ttwosidedscale} is then
almost identical
to that of Theorem~\ref{Trightscale}, using Proposition~\ref{Pindsides}
instead of
Proposition~\ref{prfixeddrift}. Note, however, that the expected
number of
children of a vertex on the backbone of $\cT_G$ or $\cT_D$ [i.e., $\E
(Z_n)$ or $\E(\widetilde Z_n)$] is divided by $2$ compared to the
conditioned case. As a consequence, the limits of the rescaled coding paths
of $\cT_G^{\beta}, \cT_R^{\beta}$ will be expressed in terms of solutions
to the equation
%
\begin{equation}
\cY_t^{\beta} = \sqrt{\gamma} B_t - \int_0^t L\biggl( j_0^{\beta}
-\frac{2}{\gamma} \underline{\cY}^{\beta}_{  s}\biggr) \,ds
\end{equation}
instead of the equation (\ref{eYbetaEquation}) 
from Lemma~\ref{Lbeta}.
Further details are left to the reader.


\subsection{\texorpdfstring{Convergence of trees: Proof of Theorem \protect\ref{Tlimexist}}
{Convergence of trees: Proof of Theorem 1.1}}
\label{subtrees}

In this section we prove weak convergence of the trees as 
metric
spaces. We refer to~\cite{LeGall2005}
for background on the theory of continuous
real trees.
\begin{pf*}{Proof of Theorem~\ref{Tlimexist}}
To prove convergence in the pointed Gromov--Hausdorff topology, it
suffices to prove that the ball of radius $R$ in the rescaled metric
converges in the ordinary Gromov--Hausdorff sense (note that these balls
are all compact a.s.). To simplify the argument, we will consider $\cR
$, the IPC conditioned to have its backbone on the right, which does
not affect the metric structure.

For compact real trees $T_g,T'_g$ coded by compactly supported contour
functions $g,g'$, the inequality
%
\begin{equation}\label{GHcontourbound}
d_{\mathrm{G\mbox{-}H}}(T_g,T'_g)\leq2\Vert g-g'\Vert_\infty
\end{equation}
relates convergence of contour functions to convergence of metric
spaces (see, e.g.,~\cite{LeGall2005}, Lemma 2.4). Therefore, fix $R>0$
and write
\[
g_k(t)=k^{-1} C_{\cR}(2k^2 t),\qquad
T_{k,R}=\sup\{t\dvtx g_k(t)\leq R\}.
\]
By Theorem~\ref{Trightscale}, $g_k$ converges in distribution as
$k\to
\infty$.
%
\begin{claim}\label{clTkRconv}
$T_{k,R}$ also converges in distribution.
\end{claim}

Assuming this for the moment, the function defined by
\[
g_{k,R}(t)=
\cases{
g_k(t) \wedge R, &\quad if $t\leq T_{k,R}$,\cr
R+T_{k,R}-t, &\quad if $T_{k,R}<t\leq R+T_{k,R}$,\cr
0, &\quad if $t>T_{k,R}+R$,}
\]
is continuous, has compact support, and converges in distribution as
$k\to\infty$. But $g_{k,R}$ is a contour function coding the part of
$\cR$ within rescaled distance $R$ of the root. By (\ref
{GHcontourbound}) this completes the proof subject to Claim~\ref{clTkRconv}.
\end{pf*}
\begin{pf*}{Proof of Claim~\ref{clTkRconv}}
$T_{k,R}$ is determined by $g_k(t)$, but we have convergence of
$g_k(t)$ only for $t$ in compact subsets of $\R_+$. Therefore, it
suffices to show that $T_{k,R}$ is tight.\vadjust{\goodbreak}

Fix $t>0$ and note that $\P(T_{k,R}>t)$ is the probability that the
tree $\cR$ has more than $k^2 t$ descendants of backbone vertices at
heights at most $k R$. We will bound this by replacing $\cR$ by a
stochastically larger tree $\cT$, namely, the tree $\cT^k$ from
Section~\ref{secfixeddriftproof} with $w_k=p_c$ for each $k$.
Write $\cU$ for the Lukaciewicz path for the corresponding sequence of
off-backbone paths, so that $-\underline{\cU}([k^2 t])$ is the height
of the backbone vertex from which the $[k^2 t]$th vertex is
descended. Thus, $\P(T_{k,R}>t)\leq\P(-\underline{\cU}(k^2 t)\leq k
R)$. But $-\frac{1}{k}\underline{\cU}(k^2 t)\to
-\underline{B}_{\gamma t}$, where $B_t$ is a Brownian motion.
Tightness follows since $-\underline{B}_{\gamma t}\nearrow\infty$ as
$t\to\infty$.
\end{pf*}
\section{\texorpdfstring{Level sizes and volumes: Proofs of Theorems \protect\ref{Tlevels-volume} and \protect\ref{Tlevels}}
{Level sizes and volumes: Proofs of Theorems 1.4 and 1.5}}\label{seclevels}

\mbox{}

\begin{pf*}{Proof of Theorem~\ref{Tlevels-volume}}
We first prove (\ref{eqvolume}). We begin by observing that
\[
\frac{1}{n^2} C[0,an] = \int_0^{\infty} \mathbf{1}_{[0,a]}
\biggl(\frac{1}{n}H_{\cR}(sn^2) \biggr)\,ds.
\]
Our objective is the limit in
distribution
\[
\int_0^{\infty} \mathbf{1}_{[0,a]} \biggl( \frac{1}{n} H_{\cR
}
(sn^2)\biggr) \,ds
\xrightarrow{n\to\infty}
\int_0^{\infty} \mathbf{1}_{[0,a]} (H_s) \,ds.
\]
This almost follows from Theorem~\ref{Trightscale}. The problem
is that $\int\mathbf{1}_{[0,a]} (X_s) \,ds$ is not a continuous function
of the process $X$, and this is for two reasons. First,
because of the indicator function, and second, because the topology is
uniform convergence on
compacts and not on all of $\rr$.

To overcome the second obstacle, we argue that for any $\eps$ there is
an $A$ such
that
\[
\P\biggl(\int_A^\infty\mathbf{1}_{[0,a]} \biggl(\frac{1}{n} H_{\cR
}
(sn^2)\biggr) \,ds \neq0\biggr)
<\eps.
\]
Indeed, in order for the height function to visit $[0,na]$ after time
$n^2A$, the total size of the $[na]$ subcritical trees attached to the
backbone up to height $[na]$ must be at least $[n^2A]$. This probability
is small for $A$ sufficiently large, even if the trees are replaced by
$[na]$ critical trees. Thus, it
suffices to prove that for every $A$
%
\begin{equation}\label{eqfinintconv}
\int_0^A \mathbf{1}_{[0,a]} \biggl(\frac{1}{n} H_{\cR}
(sn^2
)\biggr) \,ds
\xrightarrow{n\to\infty}^{\mathrm{dist}.}
\int_0^A \mathbf{1}_{[0,a]} (H_s) \,ds.
\end{equation}

Next we deal with the discontinuity of $\mathbf{1}_{[0,a]}$ by a standard
argument. We may bound $f_\eps\leq\mathbf{1}_{[0,a]} \leq g_\eps$, where
$f_\eps,g_\eps$ are continuous and coincide with $\mathbf{1}_{[0,a]}$
outside of $[a-\eps,a+\eps]$. Define the operators
\[
F_\eps(X) = \int_0^A f_\eps(X_s) \,ds,\qquad
G_\eps(X) = \int_0^A g_\eps(X_s) \,ds.
\]
Then we have a sandwich
\[
F_\eps\biggl(\frac{1}{n} H_{\cR}(sn^2)\biggr)
\leq\int_0^A \mathbf{1}_{[0,a]} \biggl(\frac{1}{n} H_{\cR}
(sn^2)\biggr) \,ds \leq
G_\eps\biggl(\frac{1}{n} H_{\cR}( sn^2 ) \biggr),
\]
and similarly for $H_s$. By continuity of the operators,
\[
F_\eps\biggl(\frac{1}{n} H_{\cR}(sn^2)\biggr)
\xrightarrow
{n\to\infty}^{\mathrm{dist}.} F_\eps(H_s),\qquad
F_\eps\biggl(\frac{1}{n} H_{\cR}(sn^2)\biggr)
\xrightarrow
{n\to\infty}^{\mathrm{dist}.} F_\eps(H_s).
\]
In the limit we have
\[
G_\eps(H_s)-F_\eps(H_s) \xrightarrow{\eps\to0}^{\mathrm{a.s.}} 0
\]
and since $G_\eps-F_\eps$ is continuous, we also have for any $\delta>0$
\[
\lim_{\eps\to0} \lim_{n\to\infty}
\P\biggl( G_\eps\biggl(\frac{1}{n} H_{\cR}(sn^2)\biggr)
-F_\eps
\biggl(\frac{1}{n} H_{\cR}(sn^2)\biggr) > \delta
\biggr) = 0.
\]
Combining these bounds implies (\ref{eqfinintconv}), and thus (\ref
{eqvolume}).

We now turn to the proof of (\ref{eqlevels}). From (\ref
{eqvolume}), we
know that for any $\eta>0$,
\[
\frac{1}{\eta n^2} C[an, (a+\eta)n] \xrightarrow{n\to\infty}^{\mathrm{dist}.}
\frac{1}{\eta} \int_0^{\infty}
\mathbf{1}_{[a,a+\eta]}(H_s)\,ds.
\]
Thus, (\ref{eqlevels}) will follow if we can prove that for any $\eta>0$,
we have the following limit in probability as $n\to\infty$:
%
\begin{equation}\label{eqetaapprox}
\biggl\vert\frac{\eta n C[an]-C[an,(a+\eta)n]} {\eta n^2} \biggr\vert \stackrel
{\P}{\to} 0.
\end{equation}

For a given vertex $v$, let $h_v$ denote the height of $v$. If $v$ is
not on the backbone, we let $\mbox{perc}(v)$ be the percolation
parameter of the off-backbone percolation cluster to which $v$ belongs.
We now single out the vertex on the
backbone at height $[an]$ and group together vertices at height $[an]$
which correspond to the same percolation parameter.

More precisely, if $\widehat w_1, \widehat w_2, \widehat w_3,\ldots,
\widehat w_{N_n}$ are the distinct
values taken by the $\widehat W$-process up to time $[na]$, we let
\[
C_n^{(w_i)}:= \{ v \in\IPC\setminus\BB\dvtx h_v =[an], \mbox
{perc}(v) = \widehat w_i \},
\]
so that
\[
\mathfrak{C}[an]:= \{ v \in\IPC\dvtx h_v = [an] \} =
\bigcup_{i=1}^{N_n} C^{(\widehat w_i)} \cup\BB_{[an]},\qquad
C[an] = \# \mathfrak{C}[an].
\]

Moreover, any vertex between heights $[an]$ and $[(a+\eta)n]$ in the
IPC descends from one of the vertices of $\mathfrak{C}[an]$. We
let
\begin{eqnarray*}
\cP_n^{(\widehat w_i)} &:=& \bigl\{v \in\IPC\setminus\BB\dvtx [an]
\le h_v
\le(a+\eta)n,  \exists w \in C^{(\widehat w_i)} \mbox{ s.t. } w \le v
\bigr\},\\
\cP_n^{\mathrm{BB}_{[an]}} &:=& \bigl\{ v \in\IPC\dvtx [an] \le h_v \le
(a+\eta)n,
\BB_{[an]} \le v \bigr\}.
\end{eqnarray*}
In particular, $C_n^{(w_i)} \subset\cP_n^{(w_i)}$ and vertices of the
backbone between heights
$[an]$ and $[(a+\eta)n]$ are contained in $\cP_n^{\mathrm{BB}_{[an]}}$. Moreover,
\[
\mathfrak{C}[an,(a+\eta)n]:= \{ v \in\IPC\dvtx [an] \le h_v \le
(a+\eta)n \} =
\cP_n^{\mathrm{BB}_{[an]}} \cup\bigcup_{i=1}^{N_n} \cP_n^{(w_i)}.
\]
However, the number of distinct values of percolation parameters
which one sees at height $[an]$ remains bounded with arbitrarily
high probability.
%
\begin{claim} \label{clfiniteW}
For any $\epsilon> 0$, there is $A>0$ such that, for any $n \in
\N$,
\[
\P\bigl[ \# \bigl\{ i \in\{1,\ldots,N_n\}\dvtx \bigl\vert C_n^{(w_i)}\bigr\vert \neq0
\bigr\} >A
\bigr] \le\epsilon.
\]
\end{claim}

From~\cite{AGdHS2008}, Proposition 3.1,
the number of distinct
values the $\widehat{W}$-process takes between $[na]/2$ and $[na]$ is
bounded, uniformly in $n$, with arbitrarily high probability.
Furthermore, it is well known that with arbitrarily high
probability, among $[na]/2$ critical Galton--Watson trees, the number
which reach height $[na]/2$ is bounded, uniformly in $n$. It
follows that the number of clusters rising from the backbone at
heights $\{0,\ldots,[na]/2\}$ and which possess vertices at height
$[na]$ is, with arbitrarily high probability, also bounded for
all $n$. The claim follows.
%
\begin{claim} \label{clBBprogeny} For any $\eta>0$, in
probability,
\[
\lim_{n \to\infty} \biggl\vert\frac{1}{\eta n^2} \cP_n^{\mathrm{BB}_{[an]}}
\biggr\vert =0.
\]
\end{claim}

Fix $\eta$. We observe that $\cP_n^{\mathrm{BB}_{[an]}}$ is bounded by the total
progeny up to height $\eta n$ of $\eta n$ critical Galton--Watson
trees. If $\vert B\vert$ denotes a reflected Brownian motion and
$l_t^0(\vert B\vert)$ its local time at $0$ up to $t$, we then deduce
from a
convergence result for a sequence of such trees (cf. formula (7) of~\cite{LeGall2005})
that
for any $\epsilon>0$,
\[
\limsup_{n \to\infty} \P\biggl[ \frac{1}{\eta n^2}
\cP_n^{\mathrm{BB}_{[an]}} > \epsilon\biggr] \le\P\biggl[
\frac{1}{\eta} \inf\{t>0\dvtx l_t^0(\vert B\vert)> \eta\} > \epsilon
\biggr],
\]
and the claim follows from the fact that $( \inf\{ t>0\dvtx
l_t^0(\vert B\vert) > u \}, u \ge0 )$ is a half stable subordinator.
%
\begin{claim} \label{clpieces}
For any $t\in(0,a)$, $\eta>0$, in probability,
\[
\lim_{n \to\infty} \biggl\vert\frac{ \cP_n^{(\widehat
{W}_{[nt]})}}{\eta n^2} - \frac{\#(C_n^{(\widehat{W}_{[nt]})})}{n}
\biggr\vert =0.
\]
\end{claim}

Fix $t,\eta$, and define $w_n:= \widehat{W}_{[nt]}$. We have
\begin{eqnarray*}
&&
\P\biggl[ \biggl\vert\frac{ \cP_n^{(w_n)}}{\eta n^2} - \frac{\#
(C_n^{(w_n)})}{n} \biggr\vert > \epsilon\biggr] \\
&&\qquad \le
\P\bigl[ \#\bigl(C_n^{(w_n)}\bigr) > n \epsilon^{-2} \bigr] + \P
\biggl[ \biggl\vert\frac{ \cP_n^{(w_n)}}{\eta n^2} - \frac{\#
(C_n^{(w_n)})}{n} \biggr\vert > \epsilon,  \#\bigl(C_n^{(w_n)}\bigr) <
\epsilon^{2} n \biggr] \\
&&\qquad\quad{} +
\sum_{k=[\epsilon^2 n]}^{[\epsilon^{-2}
n]} \P\bigl(\#\bigl(C_n^{(w_n)}\bigr)=k\bigr) \P\biggl[ \biggl\vert\frac{ \cP
_n^{(w_n)}}{\eta n^2} - \frac{\#(C_n^{(w_n)})}{n} \biggr\vert >
\epsilon \Big\vert \#\bigl(C_n^{(w_n)}\bigr)=k\biggr].
\end{eqnarray*}

Using a comparison to critical trees as in the previous argument, the
first two terms in the sum above go to $0$ as $n \to\infty$.
Furthermore, from~\cite{DuqLeG2002}, Corollary 2.5.1, we know that,
conditionally on the processes $\widehat{W}$, $L$, for any $u>0$, the
level sets of $[un]$ subcritical Galton--Watson trees with branching
law $\operatorname{Bin}(\sigma, w_n)$ converge to the local time
process of a reflected drifted Brownian motion $(\vert X_s\vert,
s\ge0)$, with drift $L(t)$, stopped at $\tau_u$. Therefore, for any
$u>0$,
\begin{eqnarray*}
&& \lim_{n \to\infty} \P\biggl[ \biggl\vert\frac{ \cP_n^{(w_n)}}{\eta
n^2} - \frac{\#(C_n^{(w_n)})}{n} \biggr\vert >
\epsilon \Big\vert \#\bigl(C_n^{(w_n)}\bigr)= [nu] \biggr] \\
&&\qquad =
\P\biggl[ \biggl\vert\frac{1}{\eta} \int_{0}^{\tau_u} \mathbf
{1}_{[0,\eta]}(\vert X_s\vert) \,ds - l_t^0(\vert X\vert) \biggr\vert >
\epsilon
\biggr],
\end{eqnarray*}
which for any $\epsilon>0$ goes to $0$ as $\eta\to
0$. Thus, by dominated convergence,
\begin{eqnarray*}
&&\lim_{\eta\to0} \limsup_{n \to\infty}
\sum_{k=[\epsilon^2 n]}^{[\epsilon^{-2} n]} \P\bigl(\#\bigl(C^{(w_n)}\bigr)=k\bigr)
\\
&&\hspace*{57.6pt}\qquad{}\times
\P\biggl[ \biggl\vert\frac{ \cP_n^{(w_n)}}{\eta n^2} - \frac{\#
(C^{(w_n)})}{n} \biggr\vert > \epsilon \Big\vert
\#\bigl(C^{(w_n)}\bigr)=k\biggr] =0.
\end{eqnarray*}
Claim~\ref{clpieces} follows.

From our decompositions of $\mathfrak{C}[an,(a+\eta)n],
\mathfrak{C}[an]$, and
Claims~\ref{clfiniteW},~\ref{clBBprogeny} and~\ref{clpieces}, we now
deduce (\ref{eqetaapprox}). This implies (\ref{eqlevels}) and
completes the proof of Theorem~\ref{Tlevels-volume}.
\end{pf*}
\begin{pf*}{Proof of Theorem~\ref{Tlevels}}
The basis of the proof is to express the limiting quantity in (\ref
{eqlevels}) as a sum of
independent contributions corresponding to distinct excursions of
$Y-\underline{Y}$. Conditionally on the
$L$-process, these contributions will be independent exponential random
variables, with parameters arising from certain excursion measures.

From (\ref{eqlevels}), the corollary will be proved if we manage to
express $\frac{\gamma}{4} l_{\infty}^a(H)$ as the right-hand side of
(\ref{eqsumofexp}). Note\vadjust{\goodbreak} that, if $l_t^x(\frac{\sqrt{\gamma
}}{2}H)$ denotes the
local time up to time $t$ at level $x$ of
\[
\frac{\sqrt{\gamma}}{2} H = Y_t -
\frac{3}{2} \underline{Y}_t,
\]
then
\[
\frac{\gamma}{4} l_{t}^a(H) = \frac{\sqrt{\gamma}}{2} l_t^{
{\sqrt
{\gamma}a}/{2} }\biggl(
\frac{\sqrt{\gamma}}{2} H \biggr),
\]
so that we may as well express $\frac{\sqrt{\gamma}}{2} l_t^{
{\sqrt
{\gamma}a}/{2} }(
\frac{\sqrt{\gamma}}{2} H )$.

To reach this goal, it is convenient to decompose the path of
$ \frac{\sqrt{\gamma}}{2} H $ according to the excursions above the
origin of
$Y-\underline{Y}$. Let us introduce some notation. We let
$\cF(\rr_+,\rr)$ denote the space of real-valued finite
paths, so that excursions of $Y$ and of $Y- \underline{Y}$ are
elements of $\cF(\rr_+,\rr)$. For a path $e \in
\cF(\rr_+,\rr)$, we define $\overline{e}:= \sup_{s \ge0}
e(s)$, $\underline{e}:= \inf_{s \ge0} e(s)$. For $c \ge0$, we
let $N^{(-c)}$ denote the excursion measure of drifted Brownian
motion with drift $-c$ away from the origin, and $n^{(-c)}$ that of
reflected drifted Brownian motion with drift $-c$ above the origin
(see, e.g.,~\cite{RogWill1994v2}, Chapter VI.8).
%
\begin{lemma} \label{LncNc}
For any $c>0$, $a>0$, we have
%
\begin{eqnarray}
\label{eqnc}
n^{(-c)}(\overline{e} >a ) &=& \frac{2c}{\exp(2ca)-1}, \\
\label{eqNc}
N^{(-c)}(\underline{e} < -a) &=&
\frac{c}{1-\exp(-2ca)}.
\end{eqnarray}
For $c=0$ we have $n^{(0)}(\overline{e}>a)=a^{-1}$,
$N^{(0)}(\underline
{e}<-a)=(2a)^{-1}$.
\end{lemma}

This result is well known and can be proven by using basic properties
of drifted
Brownian motion and excursion measures.

We are now going to determine the excursions of $Y-\underline{Y}$ which
give a nonzero contribution to $\frac{\gamma}{4}l_{\infty}^a(H)$.
We may and will choose $-\underline{Y}$ to be the local time
process at $0$ of $Y-\underline{Y}$. Using excursion theory (see,
e.g.,~\cite{RogWill1994v2}, Section~VI.8.55),
we know that for this
normalization of local time, conditionally on the $L$-process, the
excursions of $Y-\underline{Y}$ form an inhomogeneous Poisson point
process $\mathfrak{P}$ in the space $\rr_+ \times
\cF(\rr_+,\rr_+)$ with intensity $ ds \times
n^{(-L(s))}$.

For $b \ge0$, let $\tau_{b}$ denote the hitting
time of $b$ by $-Y$. Note that for any $s>\tau_b$,
$-\underline{Y}_s > b$, from the fact that drifted Brownian motion
started at $0$ instantaneously visits the negative
half line. We therefore observe that the last visit to $\frac{\sqrt
{\gamma}}{2} a$ by $\frac{\sqrt{\gamma}}{2} H $ is at time
$\tau_{a\sqrt{\gamma}}$. Hence, any point of $\mathfrak{P}$ whose first
coordinate is larger than $a\sqrt{\gamma}$ corresponds to a part of the
path of $H$ which lies strictly above $a$, and therefore cannot
contribute to $l_{\infty}^a(H)$. Moreover, a part of the path of
$\frac{\sqrt{\gamma}}{2} H$ which corresponds to an excursion of
$Y-\underline{Y}$
starting at a time $s < \tau_{a\sqrt{\gamma}}$ will only reach height
$\frac{\sqrt{\gamma}}{2} a$ whenever the supremum of this excursion is\vadjust{\goodbreak}
greater or
equal than $\frac{1}{2}( a\sqrt{\gamma} - \underline{Y}_s)$. Therefore,
any excursion
of $Y-\underline{Y}$ which gives a nonzero contribution to
$l_{\infty}^a(H)$ corresponds to a point of $\mathfrak{P}$ whose
first coordinate is some $s$ such that $s \le a\sqrt{\gamma}$ and whose
second coordinate is an excursion $e$ such that $\overline{e} \ge
\frac{1}{2}(a\sqrt{\gamma} - s)$.

These
considerations, along with properties of Poisson point processes, lead
to the following claim.\vspace*{-2pt}
%
\begin{claim} \label{clcontributing}
Conditionally on the $L$-process, the excursions of
$Y-\underline{Y}$ which give a nonzero\vspace*{1pt} contribution to
$\frac{\gamma}{4}l_{\infty}^a(H) = \frac{\sqrt{\gamma}}{2}
l_{\infty}^{{\sqrt{\gamma}a}/{2}}(
\frac{\sqrt{\gamma}}{2} H )$ are points of a Poisson point
process $\cP\subset\mathfrak{P}$ on $\rr_+ \times
\cF(\rr_+,\rr_+)$ with intensity
\[
\mathbf{1}_{[0,a\sqrt{\gamma}]}(s) \mathbf{1}
\bigl(\overline
{e}\geq\tfrac{1}{2}\bigl(a\sqrt{\gamma}-s\bigr)\bigr)\,ds\times
n^{-L(s)}(\cdot).\vspace*{-2pt}
\]
\end{claim}

The number of points of $\cP$ clearly is almost surely countable,
so we may write
$\cP= (s_i, e_i)_{i \in\zz_+}$. In particular, by
(\ref{eqnc}), $(s_i)_{i \in\zz_+}$ are the points of the Poisson
point process on $[0,a\sqrt{\gamma}]$ introduced in Theorem~\ref{Tlevels}.

Note that $\{e_i, i\in\zz_+\}$ correspond obviously to distinct
excursions of $Y-\underline{Y}$, so that their contributions to
$l_{\infty}^{{\sqrt{\gamma}a}/{2}}(
\frac{\sqrt{\gamma}}{2} H )$ are independent.\vspace*{-2pt}
%
\begin{claim} \label{clcontributions}
Conditionally given $L$, for each $i \in\zz_+$ the contribution of the
excursion $e_i$ to
$l_{\infty}^{{\sqrt{\gamma}a}/{2}}(
\frac{\sqrt{\gamma}}{2} H )$ is
exponentially distributed with parameter
\[
N^{(-L(s_i))}\bigl(\underline{e_i} \le\tfrac{1}{2}\bigl(-a\sqrt{\gamma} +
s_i\bigr) \bigr).\vspace*{-2pt}
\]
\end{claim}

Fix $i \in\zz_+$, and condition on $L$. Recall that $(s_i,e_i)$ is
one of the points of the Poisson process $\cP$, so that
$e_i$ is chosen according to the measure
\[
n^{(-L(s_i))}\bigl(  \cdot
,\overline{e} > \tfrac{1}{2}\bigl(a\sqrt{\gamma} - s_i\bigr)\bigr).
\]
Up to the time at which
$e_i$ reaches $\frac{1}{2}(a\sqrt{\gamma} -s_i)$, $e_i$ does not
contribute to
$l_{\infty}^{{\sqrt{\gamma}}a/{2}}(
\frac{\sqrt{\gamma}}{2} H )$. From the Markov property of $e$
under the restricted measure
$n^{(-L(s_i))}(  \cdot,\overline{e} > \frac{1}{2}(a\sqrt
{\gamma
} - s_i))$, the
remaining part of $e_i$ [after it has reached $\frac{1}{2}(a\sqrt
{\gamma
} -s_i$)]
follows the path of a drifted Brownian motion, with drift $-L(s_i)$,
started at $\frac{1}{2}(a\sqrt{\gamma} -s_i)$, and stopped when it gets
to the
origin. Thus, the contribution of $e_i$ to
$l_{\infty}^{{\sqrt{\gamma}a}/{2}}(
\frac{\sqrt{\gamma}}{2} H )$ is exactly the local time of this
stopped drifted Brownian motion at level $\frac{1}{2}(a\sqrt{\gamma}
-s_i)$. By
shifting vertically, it is also $l_{\infty}^0(X)$, the total local
time at the origin of $X$, a drifted Brownian motion, with drift
$-L(s_i)$, started at the origin and stopped when reaching
$\frac{1}{2}(-a\sqrt{\gamma} +s_i)$. By excursion theory, if $\widetilde
{\mathfrak{P}}_i$
is a Poisson point process on $\rr_+ \times\cF(\rr_+,\rr)$
with intensity $ds \times N^{(-L(s_i))}$, then $l_{\infty}^0(X)$ is
the coordinate of the first point of $\widetilde{\mathfrak{P}}_i$ which
falls into the set
\[
\rr_+ \times\bigl\{ e \in
\cF(\rr_+,\rr)\dvtx \underline{e} < \tfrac{1}{2}\bigl(-a\sqrt{\gamma} +s_i\bigr)
\bigr\}.
\]
Claim~\ref{clcontributions} follows.

From Lemma~\ref{LncNc}, Claim~\ref{clcontributing} (along with the
remark which follows it) and Claim~\ref{clcontributions}, we deduce
Theorem~\ref{Tlevels}.\vadjust{\goodbreak}
\end{pf*}



\printaddresses

\end{document}